\documentclass[11pt,reqno]{amsart}
\usepackage[a4paper, margin=1in]{geometry}
\usepackage{amsmath, amsthm, amssymb}
\usepackage{graphicx, array}
\usepackage{longtable}
\usepackage{epstopdf}
\usepackage{xcolor}
\usepackage{tikz}
\usepackage{cite}
\usetikzlibrary{arrows.meta}
\usepackage{wrapfig}
\usetikzlibrary{calc}
\usetikzlibrary{backgrounds}
\usetikzlibrary{automata}
\usetikzlibrary{positioning}
\usepackage{scalerel}
\usepackage{enumitem}
\usepackage[symbol, perpage]{footmisc}

\makeatletter 
\def\namedlabel#1#2{\begingroup
    #2%
    \def\@currentlabel{#2}%
    \label{#1}\endgroup
}
\makeatother

\newtheorem{Thm}{Theorem}[section]
\newtheorem{Prop}[Thm]{Proposition}
\newtheorem{Lem}[Thm]{Lemma}
\newtheorem{Cor}[Thm]{Corollary}

\newtheorem{Ques}[Thm]{Question}
\newtheorem*{Thm*}{Theorem}

{\theoremstyle{definition}
\newtheorem{Def}[Thm]{Definition}}
{\theoremstyle{definition}
\newtheorem{Rem}[Thm]{Remark}}
{\theoremstyle{definition}
\newtheorem{Ex}[Thm]{Example}}

\newenvironment{Proof}{\rm \trivlist\item[\hskip \labelsep{\bf
Proof.\quad}]}{\hfill\qed\par\medskip\endtrivlist}

\newcommand{\id}{\mathrm{id}}

\newcommand{\supp}{\operatorname{supp}}
\newcommand{\ol}{\overline}
\renewcommand{\hat}{\widehat}
\renewcommand{\epsilon}{\varepsilon}
\newcommand{\Inv}{\operatorname{Inv}}
\newcommand{\Gp}{\operatorname{Gp}}
\newcommand{\FIM}{\operatorname{FIM}}
\newcommand{\FG}{\operatorname{FG}}
\newcommand{\Lin}{\operatorname{Lin}}

\newcommand{\m}{\mathfrak m}

\newcommand{\Sg}{S\Gamma}

\newcommand{\lrel}{\mathrel{\mathcal{L}}}
\newcommand{\rrel}{\mathrel{\mathcal{R}}}

\numberwithin{equation}{section}

\begin{document}
\title[The large scale geometry of inverse semigroups]{The large scale geometry of inverse semigroups and their maximal group images}

%

\maketitle

\begin{center}

\small
MARK KAMBITES\footnote{\textit{Email address:} \texttt{Mark.Kambites@manchester.ac.uk}. This author's research was supported by the Engineering and Physical Sciences Research Council
[grant number EP/Y008626/1].} AND
N\'{O}RA SZAK\'{A}CS\footnote{\textit{Email address:} \texttt{Nora.Szakacs@manchester.ac.uk}.}

\medskip

\textit{Department of Mathematics, University of Manchester, \\
Manchester M13 9PL, UK.}

\end{center}

\begin{abstract}
The geometry of inverse semigroups is a natural topic of study, motivated both from within semigroup theory
and by applications to the theory of non-commutative $C^*$-algebras. We study the relationship between the
geometry of an inverse semigroup and that of its maximal group image, and in particular the geometric \textit{distortion} of the natural map from the former to the latter. This turns out to have both implications for semigroup theory and potential relevance for operator algebras associated to inverse semigroups. Along the way, we also answer a question of
Lled\'{o} and Mart\'{i}nez by providing a more direct proof that an $E$-unitary inverse semigroup has Yu's Property A if its maximal group image does.
\end{abstract}

\section{Introduction}

Inverse semigroups are von Neumann regular semigroups in which idempotents commute; they form the natural algebraic model of partial symmetry, and as such arise naturally in many parts of mathematics and physics \cite{Law}. A major source of modern applications is in operator algebras, where constructing $C^*$-algebras from inverse semigroups is a major theme. 

The \emph{geometry} of inverse monoids and semigroups has been a topic of much recent interest \cite{Finn, GSSz, LM21, DMthesis, CMSz, R23} both within semigroup theory and in functional analysis.
The
approaches coming from these two directions both focus on the geometric structure of \textit{Sch\"utzenberger graphs} (the strongly connected components of Cayley graphs), but are in other respects quite different in flavour. Approaches motivated from within semigroup theory tend to focus on finitely generated semigroups, with the Sch\"utzenberger graphs considered either as purely combinatorial objects or equipped with a \textit{word metric} (sometimes considered up to \textit{quasi-isometry}). Notable recent results in this area include Gray's celebrated proof of the undecidability of the word problem for special inverse monoids \cite{Bob}, and Gray and the first author's characterisation of the maximal subgroups of finitely presented special inverse monoids as precisely the recursively presented groups \cite{GK}.

For recent analytically-motivated applications, the Sch\"utzenberger graphs have been studied up to \emph{coarse equivalence}, in other words, up to bounded distortion of the metric, at which level much of the detail essential to the semigroup-theoretic approaches is not visible. In this setting a finite generation assumption is not usually required, with most results holding for countable or even quasi-countable inverse semigroups (those generated by their idempotents plus a countable set of other elements). 
In particular,
Chung, Mart\'inez and the second author  \cite{CMSz} have shown that quasi-countable inverse semigroups can be equipped with what we will here term a \textit{standard metric} (a left subinvariant, proper metric in their terminology),
in a way which is unique up to coarse equivalence, extending the well-known analogous result about countable groups \cite{DS}.
Further recent developments in this direction include work of authors including Chung, Lled\'{o}, Mart\'{i}nez and the second author \cite{LM21, CMSz}, establishing correspondences between algebraic and coarse geometric properties of inverse semigroups, and amenability- and finiteness-type properties of an associated $C^*$-algebra (the \textit{uniform Roe algebra}). 

A recurring idea in both approaches is the relationship between the geometry of the inverse semigroup and that of
its \textit{maximal group image} (the unique maximal group onto which it homomorphically maps). The natural homomorphism from an inverse semigroup to its maximal group image induces a graph
morphism from each Sch\"utzenberger graph of the former to the Cayley graph of the latter; in the important
special case of \textit{$E$-unitary} inverse monoids, these morphisms are actually embeddings. In general, even in the $E$-unitary case, these morphisms are very far from being isometries, with distances being subject to arbitrary high \textit{distortion}. 
It turns out that this distortion is important from both a semigroup-theoretic and an analytic point of view, and our
aim in this paper is to initiate the detailed study of it. In particular we look at conditions under which the
distortion is, or is not, bounded in various ways, and consequences which can be deduced from different kinds of bounds on the distortion.


If the distortion is uniformly bounded across all Sch\"utzenberger graphs, then certain geometric properties which are \textit{coarse invariant} and inherited by subspaces (such as for example Yu's \textit{Property A}) can be lifted from the maximal group image to the inverse semigroup. For the special class of \textit{$E$-unitary} inverse monoids, it was shown in \cite{LM21} that Property A lifts in this way even without assuming any bound on the distortion. The proof given in \cite{LM21} was rather indirect, proceeding via $C^*$-algebras; having briefly recapped relevant background in Section~\ref{sec:geometry}, in Section~\ref{sec:propA} we digress slightly from our main purpose to answer a question posed in \cite{LM21} by giving an elementary direct proof of this fact. In fact we establish the stronger statement that Property A lifts through any inverse semigroup morphism which is  uniformly bounded-to-one when restricted to $\mathcal{R}$-classes.

We begin Section~\ref{sec:sigma} by giving examples to show just how far the distortion can be from bounded, even for
the extraordinarily restricted class of one-relator $E$-unitary special inverse monoids. We then give an algebraic characterisation of those inverse semigroups where the distortion of the map to the maximal group image is uniformly bounded, in terms of a weak version of the \textit{$F$-inverse condition} which in its stronger form is widely studied in inverse semigroup theory. Indeed for the important class of \textit{special inverse monoids} (those admitting inverse monoid presentations where all relations take the form $w=1$), we obtain an even stronger result: here the distortion is uniformly bounded exactly if the monoid is $F$-inverse. We also consider some geometric conditions on the Sch\"utzenberger graph of $1$ which suffice to ensure that an inverse monoid is $F$-inverse.

In Section~\ref{sec:computability} we turn our attention to \textit{decidability} questions in (typically finitely generated) inverse
semigroups, and their relationship to distortion. One would not expect bounded group distortion, in the sense considered above, to have much consequence for the existence of algorithms, since (as we demonstrate with an example) distortion may still not be bounded by any recursive function. However, it is natural to ask if there are algorithmic implications to the existence of a \textit{recursive} bound on group distortion. We obtain effective versions of some of the results from Section~\ref{sec:sigma}, for example showing that an effective version of the $F$-inverse property implies a recursive bound on group distortion. A noted result of Ivanov, Margolis and Meakin \cite{IMM} states that an $E$-unitary special inverse monoid has decidable word problem provided that the maximal group image has decidable word problem and decidable \textit{prefix membership problem}; we show that in the case of recursively bounded group distortion the latter is also equivalent to decidability of the word problem in the maximal group image and of whether an element of the monoid is a right unit. (We are grateful to John Meakin, who in private communication showed us a special case of this result and in doing so suggested the idea behind the proof of the general case.) However, we use an example of Gray \cite{Bob} to show that computable group distortion does not imply decidability of the word problem (even for $F$-inverse one-relator special inverse monoids whose group images have decidable word problem). We conclude with a number
of open questions in this area.

\section{Inverse Semigroups and Their Geometry}\label{sec:geometry}

In this section we briefly recap some basic facts about inverse semigroups and their geometry. Recall that
an \textit{inverse semigroup} is a semigroup such that for every element $s$ there exists a unique element
$s^{-1}$ such that $ss^{-1}s = s$ and $s^{-1} s s^{-1} = s^{-1}$; equivalently, it is a von Neumann regular semigroup
with commuting idempotents. An \textit{inverse monoid} is an inverse semigroup with an identity element. We denote the set of idempotents in an inverse semigroup $S$ by $E(S)$. The idempotents form a (meet) semilattice with respect to the semigroup operation, and in fact the corresponding  partial order extends to $S$ by putting $s \leq t$ when $s=ts^{-1}s$, or equivalently, if there exists some $e \in E(S)$ such that $s=te$; this is called the \emph{natural partial order} on $S$. The relation is compatible with both products and taking inverses.

We say that an inverse semigroup $S$ is \emph{generated by $X$}, in notation, $S=\langle X \rangle$, if there is a map $\iota \colon X \to S$ such that $S$ is generated (under the multiplication and inverse operations) by $\iota (X)$. We do not assume $\iota$ to be an embedding, nevertheless we will not make the map $\iota$ explicit in notation. We say that $S$ is \emph{quasi-generated by $X$} if $S$ is generated by $\iota (X)$ together with $E(S)$, in notation, $S =\langle X \cup E(S) \rangle$. We say $S$ is \emph{quasi-countable} if it is quasi-generated by a countable set. An \emph{inverse monoid} is an inverse semigroup that has an identity element. When we refer to (quasi-)generating sets for inverse monoids, we mean sets that (quasi-)generate $S$ as an inverse monoid.

Given an inverse semigroup $S$ generated by $X$, we define the Cayley graph of $S$ as usual: the vertex set is $S$, and the edges are of the form $s \xrightarrow{x} sx$ for $s \in S, x \in X\cup X^{-1}$. The strongly connected component of the Cayley graph containing the vertex $s$ is called the \emph{Sch\"utzenberger graph} of $s$, and is denoted by $S\Gamma(s)$. We define an equivalence relation on $S$ by $s \rrel t$ if and only if $s$ and $t$ lie in the same
Sch\"utzenberger graph, or equivalently, if and only if $sS=tS$. (The relation $\rrel$ is one of the five equivalence relations defined by Green which play a central role most aspects of semigroup theory \cite{Howie}, but in this article we shall have no need for the other relations.)
It is an important fact that $sx \rrel s$ implies $sxx^{-1}=s$, which means that the edges of Sch\"utzenberger graphs occur in inverse pairs: there is an edge from $u$ to $v$ labelled $x$ if and only if there is an edge from $v$ to $u$ labelled $x^{-1}$. It follows that the path metric on these graphs defines a metric on each $\rrel$-class. Another useful observation is that $s \rrel ss^{-1}$ for every $s \in S$. Moreover, the $\rrel$-relation can be characterised by $s \rrel t$ if and only if $ss^{-1}=tt^{-1}$. It follows that each $\rrel$-class has a unique idempotent.

Every inverse semigroup has a \emph{maximal group image}, which can be defined as the quotient by the least congruence which identifies all the idempotents. This congruence is typically denoted by $\sigma$, and we also denote the natural homomorphism from the inverse semigroup to the maximal group image by $\sigma$. Since groups have a well-developed geometric theory, a recurring idea in geometric inverse semigroup theory \cite{GSSz, LM21} is to understand when large scale geometric properties of the maximal group image lift to the inverse semigroup. In complete generality, there is not much one can expect -- for example if the inverse semigroup has a $0$, the maximal group image is trivial. An inverse monoid is called \emph{$E$-unitary} if the map $\sigma\colon S \to S/\sigma$ is injective on each $\rrel$-class. In this case, $\sigma$ induces an embedding of each Sch\"utzenberger graph $S\Gamma(s)$ into the Cayley graph of the maximal group image, taken with respect to the same generating set $X$. An equivalent characterisation of being $E$-unitary is $\sigma$ being \emph{idempotent pure}, that is, having the property that $\sigma^{-1}(1)=E(S)$. 

We briefly recall some other definitions from inverse semigroup theory which we shall need later. An inverse semigroup is called \textit{$F$-inverse} if every $\sigma$-class contains a maximum element under the natural partial order; in fact it can easily be shown that $F$-inverse semigroups are necessarily $E$-unitary inverse monoids. 
A map $\tau : S \to T$ between inverse monoids $S$ and $T$ is called a \textit{premorphism} if $\tau(1_S) = 1_T$, and for all $s, t \in S$ we have $\tau(s^{-1}) = \tau(s)^{-1}$ and
$\tau(s)\tau(t) \leq \tau(st)$, where $\leq$ denotes the natural partial order on $T$.

We will consider all finitely generated (and more generally all quasi-countable) inverse semigroups to be equipped with a natural \emph{extended} metric. An \textit{extended metric} on a set $X$ is a function $d \colon X \times X  \to [0, \infty]$ satisfying the usual properties of the metric with respect to the usual addition and partial order on $[0, \infty]$; an \textit{extended metric space} is a set equipped with an extended metric. We can define an equivalence relation on any extended metric space by putting $x \sim y$ if $d(x,y) < \infty$; the corresponding equivalence classes are called the \textit{components} of the space. We denote the ball of radius $r$ around the point $x$ by $B(x,r)$.  An extended metric space is called \textit{uniformly discrete} if there is a positive lower bound on the distance between pairs of distinct points, and has \textit{bounded geometry} if for every $r > 0 $ there is a finite upper bound on the cardinality of balls of radius $r$.

A map $f: (X,d_X) \to (Y,d_Y)$ between (extended) metric spaces is called 
 \textit{bornologous} if there is a monotonically non-decreasing function $\rho_+: [0, \infty] \to [0,\infty]$ with $\rho_+^{-1}(\infty)=\infty$ such that
$$d_Y(f(x), f(z)) \leq \rho_+(d_X(x,z))$$ for all $x,z \in X$. The map $f$ is a
\textit{coarse embedding} if there is also a monotonically non-decreasing function $\rho_- : [0, \infty] \to [0,\infty]$ with $\rho_-(r) \to\infty$ as $r \to \infty$
such that
$$ \rho_-(d_X(x,z))\leq d_Y(f(x), f(z)) \leq \rho_+(d_X(x,z))$$
for all $x,z \in X$.
A coarse embedding is called a \textit{coarse equivalence} if there exists a constant $K$ such that $d_Y(y, f(X)) \leq K$ for all $y \in Y$. Two extended metric spaces are said to be \textit{coarse equivalent} if there exists a coarse equivalence between them; this indeed defines an equivalence relation on the class of extended metric spaces. Two extended metrics on the same set are \textit{coarse equivalent} if the identity map on the set is
a coarse equivalence between the two metric space structures (which is a stronger condition than the metric spaces merely
being coarse equivalent by an unspecified map.)

\begin{Def}[standard metric on an inverse semigroup]
	\label{def:metric}
	Let $S$ be a quasi-countable inverse semigroup, and let $d \colon S \times S \rightarrow [0, \infty]$ be a uniformly discrete extended metric whose components are the $\rrel$-classes. Following \cite{CMSz}\footnote{Actually \cite{CMSz} considers right subinvariant metrics (with $\lrel$-classes as components), following the convention in operator algebras. We work with the dual, following the convention in geometric group theory.}, we say that $d$ is
	\begin{enumerate}
		\item  \textit{left subinvariant} if $d(xs, xt) \leq d(s, t)$ for every $s, t, x \in S$; and
		\item  \textit{proper} if for every $r \geq 0$ there is some finite set $F \subseteq S$ such that $y \in xF$ for every pair $x, y \in S$ with $x \neq y$ and $d(x, y) \leq r$.
	\end{enumerate}
        We define a \textit{standard metric} on $S$ to be an extended metric which is uniformly discrete, left subinvariant, proper and whose components are the $\rrel$-classes of $S$.
\end{Def}

Standard metrics are what in \cite{CMSz} are termed \textit{left subinvariant proper metrics}.
We emphasise for the avoidance of confusion that a standard metric is (except in the case that $S$ is a group) an extended metric rather than a true metric.
When $S$ is a group, quasi-countability coincides with countability, there is only one $\rrel$-class, properness coincides with bounded geometry, and left subinvariance coincides with left invariance, and so a standard metric in our sense is exactly a left invariant metric of bounded geometry on a countable group. A key result of \cite{CMSz} shows that every quasi-countable inverse semigroup can be equipped with a standard metric, which is unique up to coarse equivalence. This means that coarse invariant properties of any standard metric are invariants of the semigroup. Where we are interested only in coarse invariant properties, we may therefore sometimes
speak of \textit{the} standard metric on an inverse semigroup.

A canonical example of a standard metric is the \emph{word metric} on a finitely (quasi-)generated inverse semigroup. More precisely, if $S=\langle X \cup E(S)\rangle$ where $X$ is a finite set, then one can define the distance between two $\rrel$-related elements $s$ and $t$ to be their distance in their Sch\"utzenberger graph with respect to the path metric, while the distance between elements which are not $\rrel$-related is defined to be $\infty$. More generally, if $X$ is countable, one can instead consider a \emph{weighted word metric}, corresponding to a function $w \colon X \to \mathbb{N}$ where $w^{-1}(n)$ is finite for all $n \in \mathbb{N}$. In this setting, instead of simply taking the path metric in the Sch\"utzenberger graphs (which would not result in a proper metric), we take the length of the edges $s \xrightarrow{x} sx$ and $sx \xrightarrow{x^{-1}} s$ to be $w(x)$. More precisely, for each pair of $\rrel$-related elements $s,t \in S$, we obtain the distance
$$d_S(s,t)=\min\{w(x_1)+w(x_2)+\ldots +w(x_n) \mid sx_1^{\epsilon_1} \cdots x_n^{\epsilon_n}=t, x_i \in X, \epsilon_i=\pm1\}.$$

As for groups \cite{NY12}, standard metrics in inverse semigroups correspond to \textit{proper length functions} \cite[Definition 3.13]{CMSz}, and these are often more convenient to work with:
\begin{Def}[proper length function on an inverse semigroup]  \label{def:length-function}
	Let $S$ be a quasi-countable inverse semigroup, and let $l\colon S \to [0, \infty)$ be a function. We say that $l$ is a \textit{length function} if $\inf_{s\in S\setminus E} l(s)>0$, and for every $s,t \in S$, the following hold:
	\begin{enumerate}
		\item \label{item:length-zero} $l(s)=0$ if and only if $s \in E$;
		\item \label{item:length-star} $l(s)=l(s^{-1})$; and
		\item \label{item:length-product} $l(st) \leq l(s)+l(t)$.
	\end{enumerate}
	Moreover, we say that $l$ is \textit{proper} if for any $r \in \mathbb R^+$, the set $C_r=\{s \in S: 0< l(s) \leq r\}$ is bounded above by a finite set in the natural partial order, that is, if $C_r \subseteq FE$ for some finite subset $F \subseteq S$.
\end{Def}

It follows easily from (\ref{item:length-zero}) and (\ref{item:length-product}) above that if $s \leq t$ in the natural partial order then $l(s) \leq l(t)$; we shall make use of this fact without further comment.

It was shown in \cite[Proposition 3.15]{CMSz} that there is a one-to-one correspondence between proper length
functions and standard metrics. Indeed, if $d$ is a standard metric on $S$, then 
	$$l(s)=d(ss^{-1},s)$$
	defines a proper length function on $S$. Conversely, if $l \colon S \to [0, \infty)$ is a proper length function on $S$, then 
	$$d(s,t) =
	\left\{\begin{array}{rl}
		l(s^{-1}t)&\hbox{if } s \rrel t,\\
		\infty&\hbox{otherwise}
	\end{array}\right.$$
	is a standard metric on $S$.

The importance of Sch\"utzenberger graphs in the theory of inverse semigroups goes back to the work of Stephen, who developed a method to study the word problem in finitely presented inverse semigroups via their Sch\"utzenberger graphs. We give a very brief primer on Stephen's work \cite{S90, Sthesis}, restricting to the inverse monoid case.

Given a set $X$, an \emph{$X$-graph} is a directed graph edge-labelled by the set $X \cup X^{-1}$ where each edge labelled by $x \in X$ and has an opposite, inverse edge with label $x^{-1}$, and vice versa. In particular, both group Cayley graphs and inverse semigroup Sch\"utzenberger graphs are $X$-graphs, where $X$ is the set of generators. 

An $X$-graph is called \emph{deterministic} if no two edges with the same label have the same initial vertex. Given a not necessarily deterministic, finite $X$-graph, we can always obtain a deterministic quotient by iteratively identifying any pair of edges (as well as their respective inverse pairs) which share the same label and the same initial vertex. We call this process \emph{determinisation}. Sch\"utzenberger graphs are of course always deterministic.

Given an inverse monoid $S$ generated by $X$ and a word $w \in (X \cup X^{-1})^\ast$ we write $[w]_S$ for the element of $S$ defined by $w$. We define the Sch\"utzenberger graph of $w$ to be $S\Gamma([w]_S)$, and sometimes denote it by $S\Gamma(w)$. This is viewed as a birooted graph with marked vertices $\alpha=[ww^{-1}]_S$ and $\beta=[w]_S$.
A key result of Stephen shows that $u$ labels an $\alpha \to \beta$ path in $S\Gamma(w)$ if and only if $[u]_S \geq [w]_S$.

An \emph{inverse monoid presentation} is of the form $\Inv\langle X | R \rangle$ where $X$ is the set of generators and $R$ is a set of relations $u_i =v_i\ (i \in I)$ for some $u_i, v_i \in (X \cup X^{-1})^\ast$. For an inverse semigroup $S$ given by such a presentation and a word $w \in X\cup X^{-1}$, Stephen described a (potentially non-terminating) procedure to build the Sch\"utzenberger graph of $w$ as an inverse limit of finite graphs, which is as follows:

\begin{itemize}
	\item Let $\Lin (w)$ be the linear graph of $w$, i.e. the birooted $X$-graph consisting of a single directed path labelled by $w$ from $\alpha$ to $\beta$. Put $\Gamma_0=\Lin (w)$.
	\item Perform a \emph{$P$-expansion}: if there is a path in $\Gamma_i$ labelled by one side of a relation such that there is no parallel path in $\Gamma_i$ labelled by the other side of that relation, sew on the missing path. More precisely, if $s=t$ is a relation such that there is a path from $u$ to $v$ labelled by $s$ but no path from $u$ to $v$ labelled by $t$, define $\Gamma_i^+$ by gluing on a disjoint copy of $\Lin(t)$ to $\Gamma_i$ from $u$ to $v$. 
	
	If there are no $P$-expansions to be performed on $\Gamma_i$, the procedure terminates.
	\item Determinise $\Gamma_i^+$ to obtain $\Gamma_{i+1}$.
\end{itemize}
Stephen showed that these steps are confluent, and that the limit in a suitable sense is $S\Gamma(w)$.

We call an $X$-graph \emph{$P$-complete} if no $P$-expansions can be performed on it. By construction (or indeed by definition), any Sch\"utzenberger graph is $P$-complete. 

When the inverse monoid is $E$-unitary, and generated by the same set $X$ as its maximal group image, $\sigma$ embeds the Sch\"utzenberger graphs of $S$ into the group Cayley graph, which can be conveniently characterised as subgraphs here. The
following lemma is straightforward to prove:
\begin{Lem}
	\label{eunitaryschutz}
	Let $S$ be an $E$-unitary inverse monoid generated by a set $X$, and $w \in (X \cup X^{-1})^\ast$. Let $G$ be the maximal group image of $S$, and
	$\Gamma(G)$ the Cayley graph of $G$ with respect to the generating set $X$. Then the Sch\"utzenberger graph $S\Gamma(w)$ is isomorphic to the smallest subgraph $\Omega$ of $\Gamma$ (under containment)   such that:
	\begin{itemize}
		\item[(i)] $\Omega$ contains the path labelled $w$ starting at the identity; and 
		\item[(ii)] for each defining relation $u_i = v_i$ and each pair $p,q$ of vertices, $\Omega$ contains a path labelled $u_i$ from $p$ to $q$ if and only if
		it contains a path labelled $v_i$ from $p$ to $q$.
	\end{itemize}
\end{Lem}

An inverse monoid presentation is called \emph{special} if all relations are of the form $w=1$ for some $w \in (X \cup X^{-1})^*$. A \emph{special inverse monoid} is one that admits a special presentation. Special inverse monoids received considerable attention over recent years \cite{DG, DG2, Bob, GK, GR}, particularly the word problem, which is inherently tied to properties of the Sch\"utzenberger graphs. In special inverse monoids, $P$-expansions always attach a loop labelled by a relator $w$ around a vertex. It follows that for every $u \in (X \cup X^{-1})^\ast$, the Sch\"utzenberger graph $S\Gamma(u)$ is the graph obtained from $\Lin (u)$ by sewing on the Sch\"utzenberger graph of $1$ at every vertex, and then determinizing.

\section{Property A in $E$-unitary inverse semigroups}
\label{sec:propA}

In this section we digress slightly from our main theme of studying distortion, in order to give a  direct proof of the converse part of \cite[Theorem 4.25]{LM21}, that is, that an $E$-unitary  inverse semigroup must have Property A if its maximal group image has Property A. We in fact show a more general statement, namely that Property A lifts under any inverse semigroup homomorphism which is uniformly bounded-to-one across all $\rrel$-classes.

The key ingredient is showing that Property A lifts under \textit{weak contractions} that are bounded-to-one. 
A weak contraction between (extended) metric spaces $(X,d_X)$ and $(Y, d_Y)$ is a map $f : X \to Y$ such that $d_Y(f(x), f(y)) \leq d_X(x,y)$ for all $x, y \in X$. Notice that weak contractions are, by definition, bornologous.
We shall need the following elementary observation (which is inspired by \cite[Lemma 5.10]{CMSz}):
\begin{Prop}\label{prop:metricdescends}
Let $S$ and $T$ be quasi-countable inverse semigroups, and let $\rho : S \to T$ be a homomorphism. Then 
\begin{enumerate}
\item \label{item:metricdescends1} there exist standard
metrics on $S$ and $T$ with respect to which $\rho$ is a weak contraction; and
\item \label{item:metricdescends2} $\rho$ is bornologous with respect to every choice of standard metrics on $S$ and $T$.
\end{enumerate}
\end{Prop}
\begin{proof}
Let $X$ be a countable set of quasi-generators for $S$, equipped with some weight function $w \colon X \to \mathbb{N}$ with $w^{-1}(n)$ finite for every $n$. Then
$X$ can also be viewed as a quasi-generating set for the subsemigroup of $T$ generated by the image of $\rho$; clearly
we may extend $X$ to a countable set of quasi-generators $Y$ for the whole of $T$, and then extend $w$ to a weight function
for $Y$ (which we will also denote by $w$) preserving the property that $w^{-1}(n)$ is finite for every $n$. These weighted sets of quasi-generators induce standard metrics $d_S$ and $d_T$ on $S$ and $T$ respectively, as described in the previous section. We claim that $\rho$ is weakly contracting with respect to these standard metrics. Indeed, if $s, t \in S$ are such $d_S(s,t) < \infty$ then there exists a word $u=x_1^{\epsilon_1} \cdots x_n^{\epsilon^n} \in (X \cup X^{-1})^\ast$ such that $w(x_1)+ \ldots +w(x_n)=d_S(s,t)$ and $s [u]_S = t$ in $S$. Since $\rho$ is a homomorphism, we have also $\rho(s) [u]_T = \rho(t)$ in $T$, which means
that $d_T(\rho(s),\rho(t)) \leq d_S(s,t)$. This establishes claim (\ref{item:metricdescends1}).

Claim (\ref{item:metricdescends2}) follows from claim (\ref{item:metricdescends1}), the fact that different standard metrics on the same quasi-countable inverse semigroup are coarse equivalent \cite[Proposition 3.22]{CMSz}, and the fact (immediate from the definitions) that the composition in either order of a bornologous map and a coarse equivalence is a bornologous map.
\end{proof}

Our next aim is to establish that Property A lifts under bounded-to-one weakly contractive maps.
We use an approach to defining Property A due originally to Higson and Roe for metric spaces (see for example \cite{NY12}), which was generalised to extended metric
spaces in \cite{LM21}.

Recall that $\ell_1(X)_{1,+}$ consists of those functions $X \to \mathbb R^+$ with $\ell_1$-norm $1$. As we consider no other norms in the paper, we denote the  $\ell_1$-norm of a function $f \colon X \to \mathbb R$ by $\|f\|$. The support of $f$ is defined as the set $\supp f=\{x \in X \colon f(x)\neq 0\}$.

\begin{Def}[Property A]
\label{def:propA}
Let $(X, d)$ be a uniformly discrete extended metric space with bounded geometry. Then $X$ has \emph{Property $A$} if and only if for every $\epsilon > 0$ and $R >0$ there exists a map $\xi \colon X \to \ell_1(X)_{1,+}, x \mapsto \xi_x$, and a number $S >0$ such that 
\begin{enumerate}
\item $\|\xi_x-\xi_y\| \leq \epsilon$ if $d(x,y) \leq R$, and
\item $\supp \xi_x \subseteq B(x,S)$.
\end{enumerate}
\end{Def}

\begin{Thm}\label{thm:contractive-propA}
Suppose $(Y, d_Y)$ is a uniformly discrete extended metric space with bounded geometry, and $(X,d_X)$ is a uniformly discrete extended metric space which admits a weak contraction to $Y$ which is at most $k$-to-one on each of the components of $(X,d_X)$. If $(Y,d_Y)$ has Property A, then $(X,d_X)$ has Property A.
\end{Thm}

\begin{Proof}
Let $f\colon X \to Y$ be a weak contraction with the property above. Suppose that $Y$ has property $A$. Let $\epsilon > 0$ and $R > 0$. Then by definition there exist $S > 0$ and a map
$\xi \colon Y \to \ell_1(Y)_{1,+}, y \mapsto \xi_y$ such that each $\xi_y$ is supported on $B_Y(y,S)$ and $\|\xi_{y_1}-\xi_{y_2}\| \leq \epsilon$ if $d_Y(y_1,y_2) \leq R$.
To show that $X$ has property A, we will define an $S' > 0$ and a map $\zeta \colon X \to \ell_1(X)_{1,+}, x \mapsto \zeta_x$ such that
such that $\zeta_x$ is supported on $B_X(x,S')$ and $\|\zeta_{x_1}-\zeta_{x_2}\| \leq \epsilon$ if $d_X(x_1,x_2) \leq R$.

Since $Y$ has bounded geometry, there exists a constant $c$ such that $|B_Y(y,S)|\leq c$ for all $y \in Y$. Put $S'=k c R$.

For $p \in Y$, let $C_p=f^{-1}(B(p,S)) \subseteq X$. Notice that as $f$ is at most $k$-to-one on connected components, for every connected component
$X_0$ of $X$, we have 
$|C_p \cap X_0| \leq k |B(p,S)| \leq k c$, so that  $|C_p \cap X_0|R \ \leq \ k c R \ = \ S'$.

We define an equivalence relation $\sim_p$ on $C_p$ by putting $a \sim_p b$ if there exist points
$x_0, \ldots, x_n \in C_p$ with $a=x_0, b=x_n$, and $d(x_i, x_{i+1}) \leq R$ for $0 \leq i < n$.

We claim that each $\sim_p$-class has diameter at most $S'$ in $X$. Indeed, let $a, b \in C_p$ with $a \sim_p b$, and let $x_0, \ldots, x_n$ be as given by the definition of $\sim_p$. Clearly all the $x_i$ must lie in the same connected component of $X$; call it $X_0$. Without loss of generality, we may assume the $x_i$s are all distinct, which
since they all lie in $C_p \cap X_0$ implies that $n+1 \leq |C_p \cap X_0|$. Then
$$d_X(a,b)= d_X(x_0,x_n)\leq d_X(x_0,x_1)+ \cdots +d_X(x_{n-1},x_n)
\leq nR  < |C_p \cap X_0|R \leq S',$$
proving the claim.

Choose a cross-section of the $\sim_p$-classes, letting $h_p \colon C_p \to C_p$ map $x \in C_p$ to the chosen representative of its equivalence class. Since $x \sim_p h_p(x)$ for any $x \in C_p$, we have $d_X(x,h_p(x))\leq S'$ by the previous claim.

We are now ready to define $\zeta_x$ for each $x \in X$. Set

$$\zeta_x(q)=
\begin{cases}
\displaystyle\sum_{\substack{p\in Y:\  h_p(x)=q\\p \in \supp\xi_{f(x)}}}\xi_{f(x)}(p) \ \ \ \ \ \ \ \ &\hbox{if } q \in B(x,S'),\\
0 &\hbox{otherwise}.
\end{cases}$$

Clearly, $\supp\zeta_x \subseteq B(x,S')$ by definition, and $\zeta_x$ is a positive function because $\xi_{f(x)}$ is. We proceed by showing that $\zeta_x$ has norm $1$. By definition, 
$$\|\zeta_x\| \ = \ \sum_{q \in B(x,S')}\sum_{\substack{p \in Y:\  h_p(x)=q\\p \in \supp\xi_{f(x)}}}\xi_{f(x)}(p),$$
and we claim that 
$$\sum_{q \in B(x,S')}\sum_{\substack{p \in Y:\  h_p(x) \ = \ q\\p \in \supp\xi_{f(x)}}}\xi_{f(x)}(p)=\sum_{p \in \supp\xi_{f(x)}}\xi_{f(x)}(p).$$
The right hand side is just $\|\xi_x\|$ which by assumption is $1$, so the statement follows from the equality.
To see that the equality holds, first note that every summand on the right hand side occurs at most once in the sum on the left hand side, so it suffices to show that all of the terms occur. If $p \in \supp \xi_{f(x)}$, then $d_Y(p, f(x)) \leq S$, and thus $x \in f^{-1}(B(p,S))=C_p$. Therefore $h_p(x)$ is defined, and as previously observed, $d_X(x, h_p(x)) \leq S'$, so putting $q=h_p(x)$ we have $q \in B(x, S')$, and $\xi_{f(x)}(p)$ occurs on the left hand side indeed.

What remains to be shown is that $\|\zeta_x-\zeta_z\| \leq \epsilon$ if $d_X(x,z) \leq R$ for $x,z \in X$. Since $d_Y(f(x),f(z))\leq d_X(x,z)<R$, we have $\|\xi_{f(x)}-\xi_{f(z)}\| \leq \epsilon$, so it suffices to show that $\|\zeta_x-\zeta_z\|\leq\|\xi_{f(x)}-\xi_{f(z)}\|$. By definition we have
\begin{align*}
\|\zeta_x-\zeta_z\|&= \sum_{y\in X} \left| \zeta_x(y)-\zeta_z(y)  \right| = 
\sum_{y\in X}\left|\displaystyle\sum_{\substack{p\in Y:\  h_p(x)=y\\p \in \supp\xi_{f(x)}}}\xi_{f(x)}(p)-\displaystyle\sum_{\substack{p\in Y:\  h_p(z)=y\\p \in \supp\xi_{f(z)}}}\xi_{f(z)}(p)\right|.\\
\end{align*}
Notice that if $p \in \supp \xi_{f(x)} \cap \supp \xi_{f(z)}$, then as before we have $x,z \in C_p$, and $d(x,z) \leq R$ implies $x \sim_p z$ by the definition of $\sim_p$, hence $h_p(x)=h_p(z)$. So following on from the previous calculation, we have
\begin{align*}
&\|\zeta_x-\zeta_z\|\leq \\
&\leq \sum_{y\in X}\left(
\sum_{\substack{p \in \supp\xi_{f(x)}\\ p \in \supp\xi_z\\h_p(x)=y\\h_p(z)=y}} \!\!\!\!\left|\xi_{f(x)}(p)-\xi_{f(z)}(p)\right|
+ \sum_{\substack{p \in \supp\xi_{f(x)}\\ p \notin \supp\xi_{f(z)}\\h_p(x)=y}}\left|\xi_{f(x)}(p)\right|
+ \sum_{\substack{p \notin \supp\xi_{f(x)}\\ p \in \supp\xi_{f(z)}\\h_p(z)=y}}\left|\xi_{f(z)}(p)\right|
\right ) \\
&=
\sum_{y\in X}\left(
\sum_{\substack{p \in \supp\xi_{f(x)}\\ p \in \supp\xi_{f(z)}\\h_p(x)=y\\h_p(z)=y}} \!\!\!\!\!\!\!\left|\xi_{f(x)}(p)-\xi_{f(z)}(p)\right|
+ \!\!\!\!\sum_{\substack{p \in \supp\xi_{f(x)}\\ p \notin \supp\xi_{f(z)}\\h_p(x)=y}}\!\!\!\!\!\!\left|\xi_{f(x)}(p)-\xi_{f(z)}(p)\right|
+ \!\!\!\!\sum_{\substack{p \notin \supp\xi_{f(x)}\\ p \in \supp\xi_{f(z)}\\h_p(z)=y}}\!\!\!\!\!\!\left|\xi_{f(x)}(p)-\xi_{f(z)}(p)\right|
\right ) \\
&=\sum_{p \in Y} | \xi_{f(x)}(p)-\xi_{f(z)}(p)|=\|\xi_{f(x)}-\xi_{f(z)}\|,
\end{align*}
as needed.

\end{Proof}

We apply the theorem above for morphisms between inverse semigroups.

\begin{Thm}
	Let $S$ and $T$ be quasi-countable inverse semigroups equipped with standard metrics. Suppose there is a morphism $\varphi \colon S \to T$ and a constant $k$ such that $\varphi$ is at most $k$-to-one on each
	$\mathcal{R}$-class of $S$. If $T$ has Property $A$, then $S$ has Property $A$.
\end{Thm}
\begin{proof}
        By Proposition~\ref{prop:metricdescends} there are standard metrics $d_S$ and $d_T$ on $S$ and $T$ with respect to which the morphism
        $\varphi$ is a weak contraction. Since Property A is coarse invariant, $T$ may be assumed to have Property A with respect to
        the standard metric $d_T$. Now by Theorem \ref{thm:contractive-propA}, $S$ has Property A with respect to the standard metric
        $d_S$, and hence with respect to every standard metric.
\end{proof}

In particular, applying to the case where $S$ is a finitely quasi-generated $E$-unitary inverse semigroup and $T$ its maximal group image, we obtain 
a geometric proof of the implication (4) $\implies$ (1) of \cite[Theorem 4.25]{LM21}, thereby answering the question the authors pose after the proof. 
We remark that the authors in \cite{LM21} study a more restrictive class of inverse semigroups (which they call \textit{finitely labelable}) than quasi-countable, as \cite{LM21} predates the introduction of what we call here standard metrics. We state the result in the full generality of quasi-countable inverse semigroups.

\begin{Cor}
	\label{cor:Diegos-question}
	If $S$ is a quasi-countable E-unitary inverse semigroup whose maximal group image has Property $A$, then $S$ has Property A.
\end{Cor}

\section{Distortion and (weak) $F$-Inverse properties}
\label{sec:sigma}

In this section we return to our main theme of the distortion of the map between the geometry of an inverse monoid and of its maximal group image. Throughout this section, $S$ will always denote an inverse semigroup, $G$ its maximal group image, and $\Gamma(G)$ the Cayley graph of $G$ (with respect to a common generating set $X$ of $S$ and $G$). Recall that the natural map from $S$ to $G$ induces a graph morphism from each Sch\"utzenberger graph $S\Gamma(s)$ of $S$ to $\Gamma(G)$ of the latter. Being morphisms (in the $E$-unitary case, embeddings) of labelled graphs, these maps are always weak contractions with respect to any weighted path metric, and therefore there is a bound on the extent to which they \textit{increase} distances in any standard metric. However, we will shortly see that the extent to which they \textit{contract} distances can be unbounded; this leads us to make the following definition:
  
\begin{Def}[bounded group distortion]\label{def:boundeddistortion}
Let $S$ be a quasi-countable inverse semigroup and $G$ its maximal group image, equipped with standard metrics $d_S$ and $d_G$ respectively. We say that $S$ has \textit{locally bounded group distortion} if for every $\mathcal{R}$-class $R$ there exists a function $\phi : [0,\infty) \to [0,\infty)$ such that
$$d_S(s,t) \leq \phi(d_G(s\sigma, t\sigma)) \textrm{ for all } s,t \in R.$$
We say that $S$ has \textit{uniformly bounded group distortion} if it has locally bounded group distortion with
the same choice of $\phi$ for every $\mathcal{R}$-class.
\end{Def}
We remark that in the case that $S$ is finitely generated and equipped with a word metric, it is equivalent to consider in the
definition functions $\phi : \mathbb{N} \to \mathbb{N}$. We also note that one may, in the definition, require without loss of
generality that the function $\phi$ be increasing. Indeed, if there is a $\phi$ satisfying the given inequalities then clearly we
may assume $\phi$ is supported only on the image of $d_G$. Now since $d_G$ is assumed to be proper its image intersects each
bounded set in finitely many points, so $\phi$ is supported on only finitely many points in each bounded set. It follows easily that
$\phi$ is bounded above by some increasing function, which we may use to replace it.

The
following proposition will allow us (among other things) to conclude that the definitions above are independent of the particular standard metrics chosen. The reader may prefer always to think of the metrics being weighted path metrics with respect to countable generating sets, or even to restrict attention to the finitely generated case and assume the
metric to be the true path metric in the Sch\"utzenberger graphs.

\begin{Prop}\label{prop:distortioncoarse}
A quasi-countable inverse semigroup has locally bounded group distortion if and only if the map $\sigma$ to the maximal group
image restricts to a coarse embedding on every $\mathcal{R}$-class. It has uniformly bounded group distortion if and only if $\sigma$ coarsely embeds each $\rrel$-class into the maximal group image with respect to the same choice of functions $\rho_+$ and $\rho_-$.
\end{Prop}
\begin{proof}
Let $S$ be an inverse semigroup, $G$ its maximal group image and $\sigma : S \to G$ the natural map.

Assume first that $S$ has uniformly bounded group distortion, with $\phi$ being the bounding function. By the remark under
Definition~\ref{def:boundeddistortion} we may assume that $\phi$ is increasing.
By Proposition~\ref{prop:metricdescends} the map $\sigma$ is bornologous
and so there exists a function $\rho_+ \colon [0, \infty] \to [0, \infty]$ with $\rho_+^{-1}(\infty)=\infty$ such that
$d_G(s \sigma, t \sigma)\leq \rho_+(d_S(s,t))$ for any $s, t \in S$, and in particular for any pair of $\rrel$-related elements $s,t \in S$.
What remains is to find a function $\rho_- : [0,\infty] \to [0,\infty]$ which tends to infinity and such 
$\rho_-(d_S(s,t)) \leq d_G(s \sigma, t \sigma)$ for every pair of $\rrel$-related elements $s, t \in S$.
We define $\rho_-$ by
$$\rho_-(k) = \inf \lbrace p \in \mathbb{R} \mid k \leq \phi(p) \rbrace.$$
It is easy to see that $\rho_-$ is weakly increasing, and since $\phi$ is increasing by assumption, we must have  $\rho_-(k) \geq p$ whenever $k > \phi(p)$, so $\rho_-$ tends to infinity. Now for every $r \in \mathbb R$ we have
$$\rho_-(\phi(r)) = \inf \lbrace p \mid \phi(r) \leq \phi(p) \rbrace \leq r,$$
since the set $\lbrace p \mid \phi(r) \leq \phi(p) \rbrace$ contains $r$. For every $s,t \in S$ with $s \rrel t$, by the defining property of
$\phi$ we have $d_S(s,t) \leq \phi(d_G(s\sigma, t\sigma))$, so applying the weakly increasing function $\rho_-$ to both sides we obtain
$$\rho_-(d_S(s,t)) \leq \rho_-(\phi(d_G(s\sigma, t\sigma))) \leq d_G(s \sigma, t\sigma)$$
as needed.

Conversely, assume that $\sigma$ is a coarse embedding on each $\rrel$-class with respect to uniformly chosen functions $\rho_-$ and $\rho_+$, that is we have $\rho_-(d_S(s,t)) \leq d_G(s \sigma, t \sigma)$ for any pair of $\rrel$-related elements $s, t \in S$. Now since $\rho_-$ tends to infinity we may define $\phi$ by
$$\phi(r) = \sup \lbrace p \in \mathbb{R} \mid \rho_-(p) \leq r \rbrace$$
whereupon $\phi(\rho_-(k)) \geq k$ for all $k \in \mathbb{R}$. Moreover, $\phi$ is easily seen to be weakly increasing, so  for every $s, t \in S$ with $s \rrel t$ we have
$$d_S(s,t) \leq \phi(\rho_-(d_S(s,t)) \leq \phi(d_G(s\sigma, t\sigma))$$
as required to show that then $S$ has uniformly bounded group distortion.

An almost identical argument applies in the locally bounded case.
\end{proof}

We remark that there is a subtle, but significant difference between $\sigma$ being a coarse equivalence (which implies that $S$ is a group) and $S$ having uniformly bounded group distortion, the former corresponding to an inequality $\rho_-(d_S(s,t)) \leq d_G(s \sigma, t \sigma) \leq \rho_+(d_S(s,t))$ for all pairs of $s,t$, and the latter only satisfying this when $s \rrel t$. 

Since the characterisations given above are both clearly invariant under coarse equivalence, and standard metrics are unique up to coarse equivalence \cite{CMSz}, a consequence of the above is the following:
\begin{Cor}\label{cor:boundeddistortionindependent}
The definitions of locally and uniformly bounded group distortion are independent of the
choice of standard metrics.
\end{Cor}

We can also characterise bounded group distortion in terms of proper length functions, which is the formulation we most often use:
\begin{Prop}
	\label{prop:length-function-distortion}
	Let $S$ be a quasi-countable inverse semigroup and $G$ its maximal group image, equipped with proper length functions $l_S$ and $l_G$ respectively. Then $S$ has locally bounded group distortion if and only if for every $\mathcal{D}$-class $D$ of $S$ there exists a function $\phi : [0,\infty) \to [0,\infty)$ such that for every $s \in D$
	we have $l_S(s) \leq \phi(l_G( s\sigma))$.
	$S$ has uniformly bounded group distortion if it satisfies the above condition uniformly with the same function $\phi$ for every $\mathcal{D}$-class.
\end{Prop}
\begin{proof}
	Let $d_S$ and $d_G$ be the standard metrics corresponding to $l_S$ and $l_G$ as discussed below Definition~\ref{def:length-function}. Suppose $S$ has locally bounded group distortion. Fix a
	$\mathcal{D}$-class $D$ and choose an $\mathcal{R}$-class $R$ in $D$, and a function
	$\phi : [0,\infty) \to [0,\infty)$ such that $d_S(x,y) \leq \phi(d_G(x\sigma, y\sigma))$ for all $x,y \in R$.
	Let $s \in D$ and denote the $\rrel$-class of $s$ by $R_s$; we aim to show that $l_S(s) \leq \phi(l_G(s\sigma))$.	 Then by Green's Lemma \cite[Theorem~2.3]{ClifPres}, we may choose $p,q \in S$ such
	that the left translations
	$$\begin{array}{cc}
		\lambda_p \colon R_s \to R & x \mapsto px\\
		 \lambda_q \colon R \to R_s & y \mapsto qy
	\end{array}$$
	are mutually inverse bijections between $R_s$ and $R$. Because $d_S$ is left subinvariant,
	these bijections must both be weakly contracting, which since they are mutually inverse implies that both are distance-preserving. Since $ps$ and $pss^{-1}$ lie
	in $R$, we have
	$$l_S(s) = d_S(ss^{-1}, s) = d_S(pss^{-1}, ps) \leq \phi(d_G((p ss^{-1}) \sigma, (ps) \sigma)) = \phi(d_G(1, s)) = \phi(l_G(s)),$$
	where the penultimate equality holds because the metric on $G$ is left invariant.
	
	Conversely, suppose for a $\mathcal{D}$-class $D$ there is a function $\phi : [0,\infty) \to [0,\infty)$ such that
	$l_S(s) \leq \phi(l_G( s\sigma)) \textrm{ for all } s\in D.$
	Then for any $x, y \in D$ with $x \rrel y$, we have that $x^{-1} y \in D$, and 
	$$d_S(x,y) = l_S(x^{-1} y) \leq \phi(l_G((x^{-1} y)\sigma)) = \phi(d_G(x,y)),$$
	hence $S$ has locally bounded group distortion. The uniformly bounded case follows immediately.
\end{proof}

The fact that a quasi-countable inverse semigroup can fail to have
locally bounded group distortion was observed (phrased in the language of course embeddings) in \cite{LM21}, with an example given in \cite[Proposition 5.4.6]{DMthesis}. Indeed, the following construction will allow us to show something even stronger: that \emph{every} subgraph of a group Cayley graph is the image in the maximal group image of a suitable Sch\"utzenberger graph of a suitable $E$-unitary inverse monoid. (This is already known for \textit{finite} subgraphs as a consequence
of the Margolis-Meakin expansion construction \cite{MM}, but for infinite subgraphs it does not seem to have been previously observed. In the related statement \cite[Theorem 3.24]{CMSz}, it is observed that every uniformly discrete, bounded metric space arises, up to bijective coarse equivalence, as the $\rrel$-class of some inverse semigroup, but the construction yields a non-$E$-unitary inverse semigroup.)

\begin{Thm}
	\label{thm:Schgraphs-main}
	Let $G=\langle X \rangle$ be a finitely generated group, and $\Delta$ an arbitrary connected subgraph of its Cayley graph. There exists an inverse semigroup $S$ generated by $X$ and a new symbol $e$ such that 
	\begin{enumerate}
		\item $S/\sigma=G$;
		\item $S$ is E-unitary; and
		\item $S$ has a Sch\"utzenberger graph which is isomorphic to $\Delta$ with an added $e$-edge labelling a loop around a chosen vertex.
	\end{enumerate}
\end{Thm}

\begin{Proof}
	Since group Cayley graphs are homogeneous, we may assume without loss of generality that the vertex of $\Delta$ chosen for an $e$-loop is the identity
	element of $G$; indeed if not, say the chosen vertex is $v$, then we may replace $\Delta$ with its isomorphic translate $v^{-1} \Delta$.
	
	We define $S$ by a presentation. The generating set is $X \cup \{e\}$ (where $e \notin X \cup X^{-1}$), and the relations are the following:
	\begin{enumerate}[label=({R\arabic*})]
		\item \label{item:e-is-id} $e^2=e$;
		\item \label{item:group-pres} $w^2=w$ for every word $w \in (X \cup X^{-1})^\ast$  with $[w]_G=1$;
		\item \label{item:paths} $e=e \ell(p)\ell(p)^{-1}$ for every finite path $p$ in $\Delta$ starting at $1$, where $\ell(p)$ denotes the label of $p$.
	\end{enumerate}
	To prove that $S/\sigma=G$, observe that the group defined by regarding the presentation above as a group presentation is indeed $G$, as \ref{item:e-is-id} ensures $[e]_G= 1$, \ref{item:group-pres} yields a presentation for $G$ by definition, and the relations in \ref{item:paths} are then automatic. The fact that $S$ is $E$-unitary follows immediately from the relations in \ref{item:e-is-id} and \ref{item:group-pres}.
	
	Now let $\Gamma$ denote the Cayley graph of $G$ with respect to the generating set $X \cup \lbrace e \rbrace$. Clearly $\Gamma$ is obtained from the Cayley graph of $G$ with respect to $X$ by adding an
	$e$-labelled loop at each vertex.  Let $\Omega$ be the subgraph of $\Gamma$ consisting of $\Delta$ with an extra $e$-labelled loop at $1$. We claim that $\Omega$
	satisfies conditions of Lemma~\ref{eunitaryschutz} to be isomorphic to the Sch\"utzenberger graph of $e$. Indeed, it clearly contains the path (in other words, the loop) labelled $e$ starting at $1$. It remains to show that if one side of a defining relation of $S$ labels a path between two vertices in $\Omega$, then the other side of the relation labels a path
	between the same two vertices. For the relations of type \ref{item:e-is-id} and \ref{item:group-pres} notice that if $w = e$ or $w \in (X \cup X^{-1})^\ast$ with $[w]_G= 1$ then every path in $\Gamma$ labelled $w$ is closed; clearly this implies that $w$ labels a path between two vertices if and only if $w^2$ does so. 
	
	For the relations \ref{item:paths}, suppose $p$ is a path in $\Delta$ starting at $1$. If $e \ell(p)\ell(p)^{-1}$ labels a path from $u$ to $v$, then since the only $e$-edge in $\Omega$ is a loop at $1$ and $\ell(p)\ell{p}^{-1}$ can only label closed paths, we must have $u=v=1$ and certainly $e$ labels a path from $1$ to $1$. Conversely, if $e$ labels a path from $u$ to $v$, then again $u=v=1$, and since $p$ is by definition a path in $\Delta$ starting at $1$, we have that $e \ell(p) \ell(p^{-1})$ labels a path from $1$ to $1$ in $\Delta$.
	
	It remains only to show that $\Omega$ is the smallest subgraph of $\Gamma$ with the given properties. Suppose $\Omega'$ were a smaller such graph; we claim that
	$\Omega'$ contains every edge of $\Omega$, and therefore is equal to $\Omega$. By definition $\Omega'$ contains the $e$-loop at $1$. Now for any other edge $q$ in $\Omega$, $q$ is also an edge of $\Delta$ and since $\Delta$ is connected we may choose a path $p$ in $\Delta$ which starts at $1$ and traverses the edge $q$. Now since $e$ labels a path from $1$ to $1$ in $\Omega'$, the conditions on the relations \ref{item:paths} imply that $e \ell(p) \ell(p)^{-1}$ labels a path from $1$ to $1$ in $\Omega'$. Since $\Omega$ and $\Omega'$ are both subgraphs of the same deterministic graph $\Gamma$, this can only be the path which traverses $q$, so $q$ must be an edge
	of $\Omega'$.
\end{Proof}

For example, we can apply Theorem~\ref{thm:Schgraphs-main} with $G = \mathbb{Z} \times \mathbb{Z}_2$, $X$ the obvious generating set, and $\Delta$ the subgraph containing of all vertices and the edges $(0,0) \to (0,1)$, $(n, 0) \to (n+1, 0)$, $(n,1) \to (n+1, 1)$ for $n \in \mathbb{Z}$. Pairs of vertices of the form $(n,0)$ and $(n,1)$ are at distance $1$ in the Cayley graph but unbounded distance in $\Delta$, so the monoid whose existence is asserted by the theorem will have group distortion which is not even locally bounded.

The presentation in Theorem \ref{thm:Schgraphs-main} (and the examples cited above) are infinite and also far from special. One might hope that by restricting the class of presentations one can ensure that distortion is controlled, but in fact even $E$-unitary special one-relator inverse monoids can fail to have even locally bounded group distortion:

\begin{Thm}
	\label{thm:onerelatorscary}
The inverse monoid $S=\Inv\langle a,b,c,d \mid bcb^{-1}ad^{-1}a^{-1} c^{-1}cd^{-1}d=1 \rangle$ is an $E$-unitary inverse monoid without locally bounded group distortion. Indeed, every
Sch\"utzenberger graph of $S$ embeds in the Cayley graph of the maximum group image with unbounded distortion.
\end{Thm}

\begin{proof}
	First, observe that the inverse monoid $S$ is the same as the one defined by the presentation $\Inv\langle a,b,c,d \mid bcb^{-1}ad^{-1}a^{-1}=1, c^{-1}c = 1, d^{-1}d=1 \rangle$. This is an idempotent-pure quotient of the one-relator inverse monoid $\Inv\langle a,b,c,d \mid bcb^{-1}ad^{-1}a^{-1}=1 \rangle$. The latter is $E$-unitary by \cite{IMM} since the relator is a cyclically reduced word, and a quotient of an $E$-unitary inverse monoid by a congruence contained in the minimum group congruence is easily seen to be itself $E$-unitary, so $S$ itself is $E$-unitary.
	
	The maximal group image of $S$ is $G=\Gp\langle a,b,c,d \mid bcb^{-1}ad^{-1}a^{-1}=1 \rangle$. If we introduce a new generator $u=bcb^{-1}$, then
	\begin{align*}
	G&=\Gp\langle a,b,c,d,u \mid bcb^{-1}ad^{-1}a^{-1}=1, u=bcb^{-1} \rangle\\
	&\cong\Gp\langle a,b,c,d,u \mid c=b^{-1}ub, d=a^{-1}ua  \rangle\\
	&\cong\FG( a,b,u ).
	\end{align*}

For convenience, we will identify the elements of $G$ with freely reduced words over the generators $a$, $b$, $u$ and their inverses. In particular, in the Cayley graph $\Gamma(G)$ of $G$ with respect to the generating set $a,b,c,d$, we identify vertices of $\Gamma(G)$ with $\FG( a,b,u )$, while the edges are labelled by the set $\{a,b,c,d\}$.
It is helpful to keep in mind this picture depicting the relations in $G$:
\begin{center}
	\begin{tikzpicture}[line width=0.8pt, >=stealth, scale=1.5]
	\foreach \i in {1, 2} {
		\foreach \j in {0, 1, 2} {
			\node (a\i\j) at (\i, -\j) {};
		}
	}
	
	\foreach \i in {1} {
		\draw[->] (a\i0) -- ++(1, 0) node[midway, above] {$c$};
		\draw[->] (a\i2) -- ++(1, 0) node[midway, below] {$d$};
		\draw[->] (a\i1) -- ++(1,0) node[midway, above] {$u$};
	}
	
	\foreach \i in {1,2} {
		\draw[->] (a\i1) -- ++(0, -1) node[midway, right] {$a$};
		\draw[->] (a\i1) -- ++(0, 1) node[midway, right] {$b$};
		
		
		\fill (1, -1) circle (1pt) node[below left] {};
	}

\end{tikzpicture}
\end{center}

The Cayley graph $\Gamma(G)$ is obtained from the Cayley graph of $\FG( a,b,u )$ (with respect to the generating set $a,b,u$) by replacing every $u$-edge with the surrounding rectangle and determinising. In particular, observe that the black edges of the graph below form a subgraph of $\Gamma(G)$:

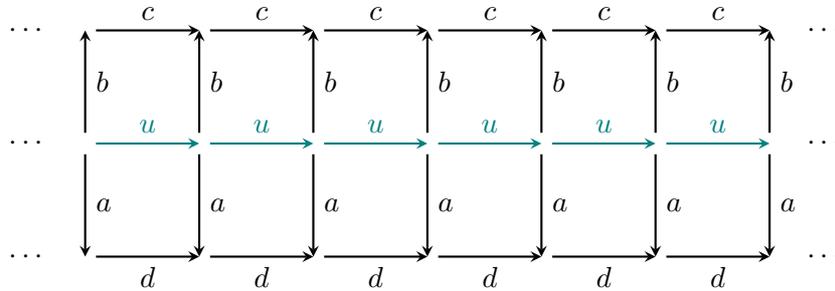
\begin{center}
	\begin{tikzpicture}[line width=0.8pt, >=stealth, scale=1.5]
		\foreach \i in {-2,-1,0, 1, 2, 3,4} {
			\foreach \j in {0, 1, 2} {
				\node (a\i\j) at (\i, -\j) {};
			}
		}
		
		\foreach \i in {-2,..., 3} {
			\draw[->] (a\i0) -- ++(1, 0) node[midway, above] {$c$};
			\draw[->] (a\i2) -- ++(1, 0) node[midway, below] {$d$};
		}
		
		\foreach \i in {-2,..., 4} {
			\draw[->] (a\i1) -- ++(0, -1) node[midway, right] {$a$};
			\draw[->] (a\i1) -- ++(0, 1) node[midway, right] {$b$};
		}
		
			\foreach \i in {-2,..., 3} {
				\draw[->, dashed, teal] (a\i1) -- ++(1,0) node[midway, above, teal] {$u$};
			}
		
		\node at (4.5, 0) {$\cdots$};
		\node at (-2.5, 0) {$\cdots$};
		\node at (4.5, -1) {$\cdots$};
		\node at (-2.5, -1) {$\cdots$};
		\node at (-2.5, -2) {$\cdots$};
		\node at (4.5, -2) {$\cdots$};

	\end{tikzpicture}
\end{center}

Let $S\Gamma(w)$ be any Schützenberger graph of $S$ (with respect to the generators $a,b,c,d$) -- we view this as a subgraph of the Cayley graph of $\Gamma(G)$.

The loop associated to the relator $bcb^{-1}ad^{-1}a^{-1} c^{-1}cd^{-1}d$ spans the following graph in $\Gamma(G)$: 

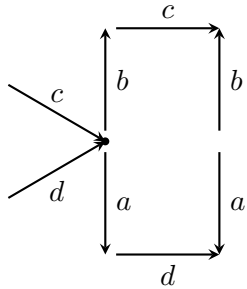
\begin{center}
	\begin{tikzpicture}[line width=0.8pt, >=stealth, scale=1.5]
	\foreach \i in {1, 2} {
		\foreach \j in {0, 1, 2} {
			\node (a\i\j) at (\i, -\j) {};
		}
	}
	
	\foreach \i in {1} {
		\draw[->] (a\i0) -- ++(1, 0) node[midway, above] {$c$};
		\draw[->] (a\i2) -- ++(1, 0) node[midway, below] {$d$};
	}
	
	\foreach \i in {1,2} {
		\draw[->] (a\i1) -- ++(0, -1) node[midway, right] {$a$};
		\draw[->] (a\i1) -- ++(0, 1) node[midway, right] {$b$};
	}
	
	\draw[->] (0.15,-0.5) -- (1, -1) node[midway, above] {$c$};
	\draw[->] (0.15,-1.5) -- (1, -1) node[midway, below] {$d$};
	
	\fill (1, -1) circle (1pt) node[below left] {};
\end{tikzpicture}
\end{center}

In particular, $S\Gamma(w)$ contains the graph below as a subgraph at every vertex $v$.

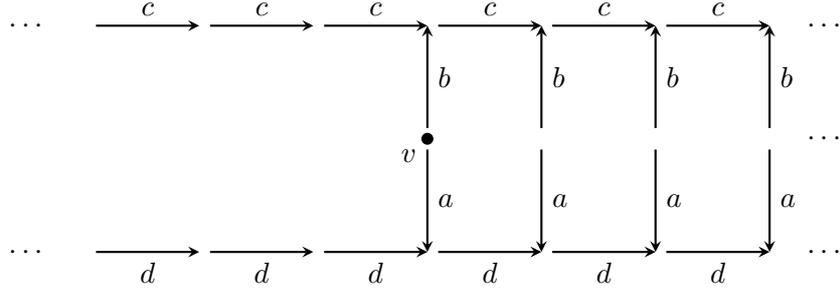
\begin{center}
	\begin{tikzpicture}[line width=0.8pt, >=stealth, scale=1.5]
	\foreach \i in {-2,-1,0, 1, 2, 3,4} {
		\foreach \j in {0, 1, 2} {
			\node (a\i\j) at (\i, -\j) {};
		}
	}
	
	\foreach \i in {-2,..., 3} {
		\draw[->] (a\i0) -- ++(1, 0) node[midway, above] {$c$};
		\draw[->] (a\i2) -- ++(1, 0) node[midway, below] {$d$};
	}
	
	\foreach \i in {1,..., 4} {
		\draw[->] (a\i1) -- ++(0, -1) node[midway, right] {$a$};
		\draw[->] (a\i1) -- ++(0, 1) node[midway, right] {$b$};
	}
	
	\node at (4.5, 0) {$\cdots$};
	\node at (-2.5, 0) {$\cdots$};
	\node at (4.5, -1) {$\cdots$};
	\node at (-2.5, -2) {$\cdots$};
	\node at (4.5, -2) {$\cdots$};
	
	\fill (1, -1) circle (1.5pt) node[below left] {$v$};
\end{tikzpicture}
\end{center}

We furthermore claim that there exists some vertex $v$ of $S\Gamma(w)$ such that no path with label of the form $bc^{-k}b^{-1}$, or $ad^{-k}a^{-1}$ $(k \in \mathbb N)$ is readable from $v$ in $S\Gamma(w)$. Note that $[bc^{-k}b^{-1}]_G=[ad^{-k}a^{-1}]_G=u^{-k}$, so it suffices to show that there exists some vertex $v$ of $\Gamma$ such that $vu^{-k} \notin V(S\Gamma(w))$ for any positive $k$.

Let $w_1, \ldots, w_n \in G$ be the vertices traversed by $w$ and define $W = \lbrace w_1, \dots w_n \rbrace$. Then for any $v \in V(S\Gamma(w))$, there is a sequence $v_1, \ldots, v_k=v$ of vertices such that $v_1 \in W$, and $v_{i+1}$ is a vertex of the subgraph of $\Gamma_G$ obtained by sewing on the relator at $v_i$. That is, in $G$, we have
$$v_{i+1} \in v_i\{a, ad, ada^{-1}, b, bc, c^{-1}, d^{-1}\}.$$ 
Rewriting in terms of the generators $a,b,u$, we have
$$v_{i+1} \in v_i\{a, ua, u, b, ub, b^{-1}u^{-1}b, a^{-1}u^{-1}a\}.$$ 
Let $k_i$ denote the number of occurrences of $u^{-1}$ in $w_i$ (viewed as a reduced word in $\FG(a,b,u)$) and let $K=-\max{k_i}$.

Assume $u^{k} \in V(S\Gamma(w))$ for some $k \leq 0$, we claim that then $k \geq K$. Then we have 
$$u^{k} \in W\{a, ua, u, b, ub, b^{-1}u^{-1}b, a^{-1}u^{-1}a\}^*.$$
Consider a product $w_i \hat w$ with $\hat w \in \{a, ua, u, b, ub, b^{-1}u^{-1}b, a^{-1}u^{-1}a\}^*$
freely reducing to $u^k$.
Notice that any occurrence of $u^{-1}$ in $w_i \hat w$ is either in $w_i$, or in a factor of $\hat w$ of the form $b^{-1}u^{-1}b$ or $a^{-1}u^{-1}a$. Thus if $k < K \leq -k_i$, then $\hat w$ must contain either a $b^{-1}u^{-1}b$- or an $a^{-1}u^{-1}a$-factor where the middle $u^{-1}$ is not removed in some reduction sequence of $w_i\hat w$ to a freely reduced word. Assume it is of the form of $b^{-1}u^{-1}b$, the other case being similar. Consider the subword of $\hat w$ following the $u^{-1}$: since the product $w_i\hat w$ freely reduces to $u^k$, this subword must also freely reduce to a power of $u$. But observe the sum of $b$-exponents in any word in $\{a, ua, u, b, ub, b^{-1}u^{-1}b, a^{-1}u^{-1}a\}^*$ is nonnegative, therefore the sum of $b$-exponents in the subword following the $u^{-1}$ is positive, which is a contradiction.

Now let $m=\min\{k\mid u^k \in V(S\Gamma(w))\} \geq K$, and put $v=u^m$. Then $u^{m-k} \notin V(S\Gamma(w))$ for any $k \geq 1$, which establishes the claim that there is a vertex $v$ such that no path of the form $bc^{-k}b^{-1}$, or $ad^{-k}a^{-1}$ $(k \in \mathbb N)$ is readable from $v$ in $S\Gamma(w)$.

Denote the path metric on $S\Gamma(w)$ by $d$, and the word metric on $G$ with respect to the generating set $\{a,b,c,d\}$ by $d_G$. For $k \geq 1$ define $u_k:=[vbc^{-k} ]_G=u^{m-k}b$ and $u_k':=[vad^{-k}]_G=u^{m-k}a$; it is easily seen that the elements $vbc^{-k}$ and $vad^{-k}$ are right-invertible in $S$ and it follows that $u_k$ and $u_k'$ are vertices of $S\Gamma(w)$. Observe that $u_kb^{-1}a=u_k'$, so $d_G(u_k, u_k')\leq 2$. We claim that $d(u_k, u_k')= 2k+2$, which would imply the that the embedding of $S\Gamma(w)$ has unbounded distortion. Indeed, we shall show that every $u_k \to u_k'$ path in $\Gamma(G)$ contains a vertex $u^n$ for some $n$. It is then clear that the shortest such path contained in $S\Gamma(w)$ is the one through $v$ with length $2k+2$.

For any path $p$ in $\Gamma(G)$, let $\ell(p)$ denote the label of $p$ as a word in the alphabet $\{a,b,c,d\}^{\pm 1}$, and let $\ell_u(p)$ be the word obtained from this label by replacing any occurrences $c^{\pm 1}$ and $d ^{\pm 1}$ with $b^{-1}u^{\pm 1}b$ and $a^{-1}u^{\pm 1}a$ respectively, that is, 
$$\ell_u(p) \in \{a,b,b^{-1}ub,a^{-1}ua,a^{-1}, b^{-1}, b^{-1}u^{-1}b, a^{-1}u^{-1}a\}^\ast.$$
Now suppose $p$ is a path from  $v_k \to v_k'$ in $\Gamma(G)$; then $\ell_u(p)$ must freely reduce to $b^{-1}a$. Furthermore, $p$ traverses a vertex $u^n \in \Gamma(G)$ if and only if $p$ has a prefix $p'$ such that $\ell_u(p')$ freely reduces to $v_k^{-1}u^n=b^{-1}u^{k-m+n}$, thus our goal is to show that such prefix exists. (Notice that not all prefixes of $\ell_u(p)$ come from prefixes of $\ell(p)$, or equivalently of $p$, so this claim is stronger than simply requiring $\ell_u(p)$ to have a prefix freely reducing to $b^{-1}u^{k-m+n}$.)

We prove the following stronger statement: if $q$ be a path in $\Gamma(G)$ such that $\ell_u(q)$ freely reduces to a word of the form $b^{-1}u^la$, then $q$ has a prefix $q'$ such that $\ell_u(q')$ freely reduces to a word of the form $b^{-1}u^n$. Note that the latter condition is equivalent to requiring that $q$ has a suffix $q''$ with $\ell_u(q'')$ freely reducing to a word of the form $u^{n'}a$. We prove the claim by induction on the length of the path $q$. The statement is obviously true when $\ell(q)=b^{-1}a$, which is the shortest path satisfying the conditions. Now let $q$ be a path such that $\ell_u(q)$ freely reduces to $b^{-1}u^la$ for some $l \geq 1$ and assume the statement is true for all shorter paths of this form. 

Fix some reduction sequence of $\ell_u(q)$ to a freely reduced word. Then we may factorise $\ell_u(q)$ as $w_1 \cdot b^{-1} \cdot w_2 \cdot a \cdot w_3$ where $w_1$ and $w_3$ freely reduce to $1$ and $w_2$ freely reduces to $u^l$. Consider the prefix $w_1 b^{-1}$. If this prefix is the label of some prefix $q'$ of the path $q$ then since it freely reduces to $b^{-1}$ we are done. Dually, if the suffix $a w_3$ is the label
of some suffix $q''$ then the complementary prefix $q'$ has label freely reducing to $b^{-1} u^l$ and we are done. Otherwise, the $b^{-1}$ ending the prefix $w_1 b^{-1}$ and the $a$ beginning the suffix $a w_3$ must come respectively from edges in $q$ labelled $b^{-1} u^{\epsilon_1} b$ and $a^{-1} u^{\epsilon_2}  a$ for some $\epsilon_1, \epsilon_2 \in \lbrace 1, -1 \rbrace$. In this case
$$\ell_u(q)=w_1 \cdot b^{-1} \underbrace{ u^{\epsilon_1} b \cdot  \hat w \cdot a^{-1} u^{\epsilon_2}  }_{w_2} a \cdot w_3.$$
Since $w_2$ freely reduces to $u^l$, $\hat w$ must freely reduce to $b^{-1}u^{l-\epsilon_1-\epsilon_2}a$. Observe that $\hat{w}$ is the label of a subpath of $q$ which satisfies the conditions of the claim and is shorter than $q$. So by the inductive hypothesis there exists a prefix $\hat q'$ of $\hat q$ such that $\ell_u(\hat q')$ freely reduces to a word of the form $b^{-1}u^n$. Let now $q'$ be the prefix of $q$ with 
$$\ell_u(q')=w_1 \cdot b^{-1}u^{\epsilon_1}b \cdot \ell_u(\hat q'),$$
which we can reduce freely as $w_1 \cdot b^{-1}u^{\epsilon_1}b \cdot \ell_u(\hat q')  \to b^{-1}u^{\epsilon_1}b \cdot b^{-1}u^n \to b^{-1}u^{\epsilon_1+n}$, proving the claim. Applying the claim to the path $p$ completes the proof of the theorem.
\end{proof}

We have seen examples of unbounded group distortion; this naturally raises the question of when the distortion \emph{is} bounded, either locally or uniformly.  Our next aim is to explore the extent to which bounded distortion can be
characterised as an algebraic property of the semigroup, related to being $F$-inverse. In an $F$-inverse monoid we write $s^\m$ for the greatest element in the $\sigma$-class of an element $s$. Notice that if $S$ is $F$-inverse and equipped with a
proper length function, then $l_S(s)\leq l_S(s^\m)$ for every $s \in S$.

It was observed in \cite[Proposition 5.2.12]{DMthesis} (in the language of coarse embeddings) that $F$-inverse monoids have uniformly bounded group distortion. In fact, the following result exactly characterises both locally and uniformly bounded group distortion as  algebraic properties
which can be seen as weak versions of being $F$-inverse.
\begin{Thm}
	\label{thm:uniformcoarse}
\begin{enumerate}
	\item Let $S$ be any quasi-countable inverse monoid. Then $S$ has uniformly bounded group distortion if and only if every $\sigma$-class contains finitely many maximal elements, and every ascending chain is bounded above.
	\item Let $S$ be any quasi-countable inverse semigroup. Then $S$ has uniformly bounded group distortion if and only if every $\sigma$-class contains finitely many maximal non-idempotent elements, and every ascending chain of non-idempotent elements is bounded above.
\end{enumerate}	
\end{Thm}

\begin{proof}
First notice that in the monoid case the idempotents are bounded above by $1$, which is therefore the only idempotent element which could be maximal in any $\sigma$-class. Hence, every $\sigma$-class contains finitely many maximal elements if and only every $\sigma$-class contains finitely many maximal non-idempotent elements. Moreover, every ascending chain either contains infinitely many idempotents (in which case it is bounded above by $1$) or does not (in which case it is bounded above if and only if the subsequence of non-idempotents is bounded above); hence, every ascending chain of elements is bounded if and only if every ascending chain of non-idempotent elements is bounded. It follows that it suffices to prove the semigroup case.

Let $G$ denote the group $S/\sigma$, and let $l_S$, $l_G$ be proper length functions on $S$ and $G$
respectively. First assume that every $\sigma$-class contains finitely many non-idempotent maximal elements, and that every ascending chain of non-idempotents is bounded above within the semigroup.
For each $g \in G \setminus \lbrace 1 \rbrace$, let $m_1^g, \ldots, m_{k_g}^g$ be the non-idempotent maximal elements of $S$ in the $\sigma$-class $g$. Since $G$ has bounded geometry we may define
$$\phi(r) = \max\{l_S(m_i^g): 1 \leq i \leq k_g,\ l_G(g)\leq r\}.$$
Since every ascending chain of non-idempotents is bounded above, every non-idempotent $s \in S$ is below some maximal element $m_s \in S$. Furthermore the idempotents are downward-closed in the natural partial order, so $m_s$ is necessarily non-idempotent and hence equal to
$m_i^{s\sigma}$ for some $i$, whereupon
$$l_S(s) \leq l_S(m_i^{s\sigma}) \leq \phi(l_G(s\sigma)).$$
On the other hand, if $s \in S$ is idempotent then
$$l_S(s) = 0 \leq \phi(0) = \phi(l_G(s\sigma)).$$
So by Proposition \ref{prop:length-function-distortion}, the group distortion is uniformly bounded by $\phi$.

Conversely assume $S$ has uniformly bounded group distortion, that is, that there is a function $\phi$ with $l_S(s) \leq \phi(l_G(s\sigma))$ for all $s \in S$. 
Choose any $g \in G$, and put $k=l_G(g)$. Then $l_S(s) \leq \phi(k)$ for all $s \in S$ with
$s \sigma = g$. Since the length function $l_S$ is proper, there is a finite subset $F \subseteq S$ such
that every non-idempotent element of $S$ of length at most $\phi(k)$ lies below an element of $F$. Since
idempotents are downward-closed in the natural partial order, we may assume $F$ contains no idempotents. Now 
since $\sigma$-classes are upward-closed, it follows that the non-idempotents
of $\sigma^{-1}(g)$ are bounded above by the finite subset $\sigma^{-1}(g) \cap F$, which is the set of non-idempotent maximal elements in  $\sigma^{-1}(g)$.
Moreover, any ascending chain of non-idempotents is contained in a $\sigma$-class, and so is bounded above by an element of the set $F$ as defined above.
\end{proof}

We can say even more in the case of $E$-unitary special inverse monoids. Given such a monoid $S=\Inv\langle X \mid w_i=1\ (i \in I) \rangle$, the vertices of the Sch\"utzenberger graph of $1$ inside $\Gamma(G)$ form a submonoid of $G$, called the \emph{prefix monoid} $P$. It is generated by all the prefixes of the relators $w_i$, and by \cite{IMM}, if $G$ has a solvable word problem, and $P$ has a decidable membership problem in $G$, then $S$ has a solvable word problem. 

We now turn our attention to $E$-unitary, special inverse monoids.
We shall need the following observation, which is likely to be of independent interest.
\begin{Prop}
\label{prop:fub-is-Finverse}
If $S$ is an $E$-unitary, special inverse monoid and $s, t\in S$ are $\sigma$-related, then they have a common upper bound in the natural partial order.
\end{Prop}

\begin{proof}
Consider the Sch\"utzenberger graphs $S\Gamma(s)$, $S\Gamma(t^{-1})$ and $S\Gamma(st^{-1})$. Let $S\Gamma(s) \vee S\Gamma(t^{-1})$ be the graph obtained from $S\Gamma(s) \sqcup S\Gamma(t^{-1})$ by identifying the vertex $s$ of $S\Gamma(s)$ with the vertex $t^{-1}t$ of $S\Gamma(t^{-1})$, and denote the identified vertex in the quotient by $z$. Note that there is a graph morphism $S\Gamma(s) \vee S\Gamma(t^{-1}) \to S\Gamma(st^{-1})$.
Furthermore, since $S$ is special, $S\Gamma(s) \vee S\Gamma(t^{-1})$ is $P$-complete. So determinizing $S\Gamma(s) \vee S\Gamma(t^{-1})$ results in a deterministic, $P$-complete graph, which must therefore be equal to $S\Gamma(st^{-1})$. 

Denote the determinisation map on the vertices by  $f\colon V(S\Gamma(s) \vee S\Gamma(t^{-1})) \to V(S\Gamma(st^{-1}))$. Our aim is to describe when two vertices are identified by $f$. Consider the relation $\sim$ on $V(S\Gamma(s) \vee S\Gamma(t^{-1}))$, where for $u \in V(S\Gamma(s)), v\in V(S\Gamma(t^{-1}))$, we define $u \sim v$ if there are is a path $z \xrightarrow{w} u$ in $S\Gamma(s)$, and $z \xrightarrow{w} v$ in $S\Gamma(t)$ for some word $w \in (X \cup X^{-1})^\ast$, and take the symmetric and reflexive closure.

We claim that $\ker f$ coincides with $\sim$. First observe that $\sim$ is contained in $\ker f$ by definition. Also
by definition, $\ker f$ is the least equivalence relation on its domain which gives a deterministic quotient, so for the
reverse inclusion is suffices to show that $\sim$ is a equivalence relation with deterministic quotient. It is clear from
the definition that $\sim$ is reflexive and symmetric, so to prove that it is a equivalence relation we need only prove that it
is transitive. Suppose $u_1 \sim u_2 \sim u_3$ for some vertices $u_1, u_2, u_3$ of $S\Gamma(s) \vee S\Gamma(t^{-1})$. Then in particular $f(u_1)=f(u_2)=f(u_3)$. By the pigeonhole principle, at least two of the three vertices must lie either in $S\Gamma(s)$ or in $S\Gamma(t^{-1})$. But since $S$ is $E$-unitary, $f$ must be injective on both $S\Gamma(s)$ and $S\Gamma(t^{-1})$, so two of the three vertices $u_1,u_2,u_3$ must coincide. Transitivity then follows.
 
It remains to show that  $S\Gamma(s) \vee S\Gamma(t^{-1})/\sim$ is deterministic. Let $u_1, v_1, u_2, v_2$ be vertices of $S\Gamma(s) \vee S\Gamma(t^{-1})$ with edges $u_1 \xrightarrow{x} v_1$ and $u_2 \xrightarrow{x} v_2$ such that $u_1 \sim u_2$. We need to show that $v_1 \sim v_2$.
If $u_1=u_2$, then we must either have $v_1=v_2$, or $u_1=u_2=z$ and $v_1 \in V(S\Gamma(s))$, $v_1 \in V(S\Gamma(t^{-1}))$, or vice versa. In either case, we must have $v_1 \sim v_2$ by definition. Otherwise, without loss of generality, we may assume that $u_1 \in V(S\Gamma(s))$, $u_2 \in V(S\Gamma(t^{-1}))$, and there is a word $w \in (X \cup X^{-1})^\ast$ with paths $z \xrightarrow{w} u_1$ and $z \xrightarrow{w} u_2$ lying in $S\Gamma(s)$ and $S\Gamma(t^{-1})$ respectively. If $u_1 =z$, then, because $S$ is $E$-unitary and $z \xrightarrow{w} u_1$ is a path in a Sch\"utzenberger graph, we must have that $w_S$ is idempotent, and $z \xrightarrow{w} u_2$ is also a cycle, making $u_1=u_2$ -- a contradiction. So we must have $u_1\neq z$, and similarly $u_2 \neq z$. It follows that $u_1 \xrightarrow{x} v_1$ is contained in $S\Gamma(s)$ and $u_2 \xrightarrow{x} v_2$ contained in $S\Gamma(t^{-1})$. But then there are paths $z \xrightarrow{wx} v_1$ in $S\Gamma(s)$ and $z \xrightarrow{wx} v_2$ in $S\Gamma(t^{-1})$, so $v_1 \sim v_2$ indeed. This completes the proof that $S\Gamma(s) \vee S\Gamma(t^{-1})/\sim$ is deterministic,
establishing the claim that $\sim$ coincides with $\ker f$.

Since $s \mathrel{\sigma} t$ and $S$ is $E$-unitary, $st^{-1}$ must be idempotent, in particular, $f$ identifies the vertex $ss^{-1}$ of $\Gamma(s)$ and the vertex $t$ of $S\Gamma(t^{-1})$, so there must be paths $z \xrightarrow{w} ss^{-1}$ in $S\Gamma(s)$, and $z \xrightarrow{w} t^{-1}$ in $S\Gamma(t^{-1})$ for some word $w \in (X\cup X^{-1})^\ast$. In particular, 
$w^{-1}$ labels a path in $S\Gamma(s)$ from $ss^{-1}$ to $s$ and hence $[w^{-1}]_S=[w]_S^{-1} \geq s$, and similarly
$w$ labels a path in $S\Gamma(t^{-1})$ from $tt^{-1}$ to $t^{-1}$ and hence $[w]_S \geq t^{-1}$ and so $[w]_S^{-1} \geq t$. This shows that $[w]_S^{-1}$ is a common upper bound for $s$ and $t$ as required.
\end{proof}

\begin{Thm}
	\label{specEunit}
A quasi-countable, $E$-unitary, special inverse monoid has uniformly bounded group distortion if and only if it is $F$-inverse.
\end{Thm}

\begin{proof}
Let $S$ be a quasi-countable, $E$-unitary, special inverse monoid. Suppose $S$ has uniformly bounded group distortion. Then
by Theorem~\ref{thm:uniformcoarse}, each $\sigma$-class has finitely many maximal elements; moreover ascending chains are
bounded, and since comparable elements lie in the same $\sigma$-class the bounds must lie in the same $\sigma$-class, so it has at least one maximal element. By Proposition \ref{prop:fub-is-Finverse} the set of maximal elements will have a common upper bound, which since comparable elements lie in the same $\sigma$-class must itself lie in the same $\sigma$-class. Thus each $\sigma$-class has exactly one maximal element and no unbounded ascending chains, from which it follows that $S$ is $F$-inverse. Conversely, if $S$ is $F$-inverse then
it certainly has the property that each $\sigma$-class has finitely many maximal elements and ascending chains are bounded above,
so by Theorem~\ref{thm:uniformcoarse} it has bounded group distortion.
\end{proof}

This makes the class of special $F$-inverse monoids a very natural one to study from a geometric perspective. Our next aim is to characterise, in terms of the presentation, when an $E$-unitary special inverse monoid is $F$-inverse. We begin with an easy sufficient condition, which does not actually require that the inverse monoid is special.

Let $S$ be an $E$-unitary inverse monoid, and denote the embedded copy of $S\Gamma(1)$ in $\Gamma(G)$ by $\Sg$. 
We will say that $\Sg$ is \emph{sprawling} if for any $g \in G$, the intersection $\Sg \cap g\Sg$ is nonempty.

\begin{Lem}
	\label{lem:sprawling}
If $S = \langle X \rangle$ is an $E$-unitary inverse monoid with a sprawling Schützenberger graph of $1$, then $S$ is $F$-inverse.
\end{Lem}

\begin{proof}
Let $g \in G$. By assumption $\Sg \cap g\Sg \neq \emptyset$, so the graph $\Sg \cup g\Sg$ is connected. Choose a $1 \rightarrow g$ path in $\Sg \cup g\Sg$, and let $w_g$ be its label. We claim that $[w_g]_S$ is the greatest element of $\sigma^{-1}(g)$. Indeed, if $s \in S$ is such that $s{\sigma}=g$, then the embedded copy of $S\Gamma(s)$ inside
$\Gamma(G)$ must contain $\Sg \cup g\Sg$, and therefore contains a path $ss^{-1} \xrightarrow{w_g} s$, proving that $[w_g]_S \geq s$.
\end{proof}

It turns out that when $S$ is special, $S\Gamma$ being sprawling can be described in terms of structural properties of the maximal group image $G$ and the prefix monoid $P$. Recall that a cancellative monoid is called \emph{right reversible} if every pair of principal right ideals intersects. If $P$ is a right reversible monoid, then it always give rise to a \emph{group of right quotients}, that is, a (necessarily unique) group $G$ containing $P$ as a submonoid where every element can be expressed in the form $pq^{-1}$ with $p,q \in P$. Conversely, every cancellative monoid with a group of right quotients is right reversible.  We refer the reader to \cite[Section 1.10]{ClifPres} for more detail.

\begin{Prop}
	\label{prop:Finverse-suff}
Let $P$ be the prefix monoid associated to the $E$-unitary special inverse monoid $S=\Inv\langle X \mid w_i=1\ (i \in I) \rangle$. Then the following are equivalent:
\begin{enumerate}
	\item\label{item:sprawling} $\Sg$ is sprawling;
	\item\label{item:Grightquotients} $G$ is the group of right quotients of $P$;
	\item\label{item:Pleftrev} $P$ is right reversible, and every generator occurs (in positive and/or negative form) in the relations.
\end{enumerate}
In particular, if any of the above conditions hold then $S$ is $F$-inverse.
\end{Prop}

\begin{proof}
We begin by showing the equivalence of (\ref{item:sprawling}) and (\ref{item:Grightquotients}).
As the prefix monoid $P$ is the vertex set of $\Sg$, we have that $\Sg$ is sprawling if and only if $P \cap gP \neq \emptyset$ for all $g \in G$. The latter holds if and only if for every $g \in G$ there exist $p, q \in P$ with $p = gq$, that is with $g = pq^{-1}$. In other words, if and only if $G = PP^{-1}$, that is, if and only if $G$ is the group of right quotients of $P$.

Now assume (\ref{item:Grightquotients}). Then $P$ is right reversible and generates $G$ as a group. In particular every generator in $X$ can be expressed in terms of elements of $P$, and therefore every generator has to occur (in positive and/or negative form) in the relations, so (\ref{item:Pleftrev}) holds.

Conversely, assume (\ref{item:Pleftrev}). Since $P$ is right reversible, the set $PP^{-1}$ forms a subgroup in $G$ (see \cite[Section 1.10]{ClifPres}). To show that this subgroup equals $G$ it will suffice to show that $P$
generates $G$ as a group. Let $x \in X$ and suppose $w_i=pxq$ for some words $p,q \in (X\cup X^{-1})^\ast$ (the case when $w_i=px^{-1}q$ is similar). Then we have $p, px \in P$ and $[x]_G=[p^{-1}]_G[px]_G$, hence $P$ generates $G$ as required. 

The fact that these conditions imply that $S$ is $F$-inverse now follows immediately from Lemma \ref{lem:sprawling}.
\end{proof}

\begin{Rem}
Notice that the equivalence of (\ref{item:sprawling}) and (\ref{item:Grightquotients}) also holds for all $E$-unitary inverse monoids if one defines $P=V(\Sg)$: the proof does not use the specialness assumption. However, if $S$ is not special, one cannot easily define $P$ in terms of relator words in a presentation.
\end{Rem}

\begin{Ex} Consider the special inverse monoid
$S=\Inv\langle a,b \mid aba^{-n}b^{-1}=1\rangle$ for some fixed $n \geq 1$. This is $E$-unitary by \cite{IMM} since the relator is cyclically reduced, and $G =Gp\langle a,b \mid aba^{-n}b^{-1}=1\rangle$ is the Baumslag-Solitar group $BS(1,n)$. It is not difficult to see that $P$ is the submonoid of $G$ generated by $a$ and $b$. Observe that $G=PP^{-1}$: for any $w \in \{a,b, a^{-1}, b^{-1}\}^*$, we can apply reductions $a^{-1}b \to ba^{-n}$ and $b^{-1}a \to a^nb^{-1}$ along with group reductions to iteratively eliminate all subwords of the form $x^{-1}y$ where $x,y$ are positive letters. This results in a word $w'$ with $[w']_G=[w]_G$ and $w'=uv^{-1}$ for some $u,v\in \{a,b\}^\ast$, showing $[w]_G \in PP^{-1}$. It follows that $S$ is $F$-inverse.
\end{Ex}

We shall need the following elementary fact about free products of $F$-inverse monoids, which may well be known but which we have been unable to find stated
in the literature. The proof uses $F$-inverse covers. Given an inverse semigroup $S$, and $F$-inverse cover of $S$ is an $F$-inverse monoid $F$ with an idempotent separating surjective morphism $F \to S$.

\begin{Lem}
	\label{lem:free-product}
	The free product of two $F$-inverse monoids (in the category of inverse monoids) is $F$-inverse.
\end{Lem}

\begin{proof}
	Let $S$ and $T$ be $F$-inverse monoids, and denote their free product by $M$. We will define a premorphism $\varphi \colon M/\sigma \to M$ such that $\sigma \circ \varphi=\id_{M/\sigma}$, that is, a map satisfying $\varphi(1)=1, \varphi(g^{-1})=\varphi(g)^{-1}$, $\varphi(g)\varphi(h)\leq \varphi(gh)$ and $\sigma(\varphi(g)) = g$ for all $g,h \in M/\sigma$. It will then follow that $M$ is $F$-inverse: indeed, given such a $\varphi$, it follows from \cite[Theorem VII.6.11]{Petrich} that
	$$F=\{(m,g)\mid m \leq \varphi(g)\} \leq M \times M/\sigma$$
	is an $F$-inverse cover of $M$, and since $\sigma \circ \varphi=\id_{M/\sigma}$, the following diagram commutes.
	\begin{center}
		\begin{tikzpicture}
		\node (T) at (0,1.5) {$F$};
		\node (M) at (2,1.5) {$M$};
		\node (G) at (1,0) {$M/\sigma$};
		
		\draw[->] (T) -- (M) node[midway, above] {$\pi_1$};
		\draw[->] (T) -- (G) node[midway, left] {$\pi_2$};
		\draw[->] (M) -- (G) node[midway, right] {$\sigma$};
	\end{tikzpicture}
	\end{center}
	Since $\pi_2$ is idempotent pure, so is the covering map $\pi_1$, and it is by definition idempotent separating and onto, therefore it must be an isomorphism.
	
	The elements of $M/\sigma=S/\sigma \ast T/\sigma$ can be uniquely represented by products $g_1g_2 \cdots g_n$ where $g_i$ belong alternately to $S/\sigma\setminus \{1\}$ or $T/\sigma\setminus \{1\}$. For any $g \in S/\sigma$ or $g \in T/\sigma$, define $g^\m$ to be the element of $S$, respectively $T$, which is greatest in the $\sigma$-class $g$. Notice that since $M$ contains isomorphic copies of both $S$ and $T$ (see \cite[Section I.13.1]{Petrich}), $g^\m$ can always be interpreted in $M$.
	We define the map $\varphi \colon  S/\sigma \ast T/\sigma \to M$ by mapping $g_1g_2 \cdots g_n$ into $g_1^\m g_2^\m \cdots g_n^\m$, where empty products as usual are interpreted as $1$.
	
	Observe that $\varphi(g)^{-1} = \varphi(g^{-1})$ is immediate from $(g^\m)^{-1} = (g^{-1})^\m$, so the only condition that requires proof is $\varphi(g)\varphi(h)\leq \varphi(gh)$. If $g=g_1 \cdots g_n$ and $h=h_1\cdots h_m$, then the unique representative of $gh$ is obtained from $g_1 \cdots g_nh_1\cdots h_m$ by iteratively replacing consecutive letters from the same group by their products, and erasing any instances of $1$. Notice that for any $g,h \in S/\sigma \cup T/\sigma$ belonging to the same group, we have $\varphi(g)\varphi(h)=g^\m h^\m \leq (gh)^\m =\varphi(gh)$, since $g^\m h^\m \mathrel{\sigma} (gh)^\m$. It follows by a formal induction that $\varphi$ is a premorphism.	
\end{proof}

Observe that the proof also tells us what the greatest element in a $\sigma$-class of $M$ is: since $F$ is isomorphic to $M$ via $\pi_1$, and the greatest element in the $\sigma$-class $g$ of $F$ is $(\varphi(g), g)$, it follows that the greatest element in the $\sigma$-class $g$ of $M$ is $\varphi(g)$.

\begin{Prop}
	\label{prop:Finverse-suff2}
Let $P$ be the prefix monoid associated to the $E$-unitary special inverse monoid $S=\Inv\langle X \mid w_i=1\ (i \in I) \rangle$. If $P$ is right reversible, then $S$ is $F$-inverse.
\end{Prop}

\begin{proof}
Denote the subset of generators occurring in the relations by $Z$, and put $Y=X \setminus Z$. Then $S$ splits as a free product of $\FIM(Y)$ and $T=\Inv\langle Z \mid w_i=1\ (i \in I)\rangle$. Free inverse monoids are well-known to be $F$-inverse, and $T$ is $F$-inverse by Proposition \ref{prop:Finverse-suff}; hence the statement follows by Lemma \ref{lem:free-product}.

The maximal group image $G$ of $S$ is the free product of $\FG(Y) $ and $H=\Gp\langle Z \mid w_i=1\ (i \in I) \rangle$. By Lemma \ref{lem:free-product} and the observation thereafter, given $g \in G$ written as an alternating product $g_1 \cdots g_n$ of non-identity elements of $\FG(Y)$ and $H$, the greatest element in the $\sigma$-class $g$ of $S$ is $g_1^\m \cdots g_n^\m$, where $g_i^\m$ is the greatest element in the $\sigma$-class $g_i$ of $\FIM(Y)$ or $T$. In particular, if $u_i$ is a representative of $g_i^\m$, then $u_1 \ldots u_n$ is a representative of the greatest element in $g$. 
\end{proof}

We now move on to give a necessary and sufficient condition for an $E$-unitary special inverse monoid $S$ to be $F$-inverse.
First, assume that $u, w\in (X\cup X^{-1})^\ast$ represent elements of $S$ in the $\sigma$-class $g$. Then $[w]_S \geq [u]_S$ if and only if the Sch\"utzenberger graph $S\Gamma(u)$ contains a $[uu^{-1}]_S \xrightarrow{w} [u]_S$ path. Consider the embedded copy of $S\Gamma(u)$ in $\Gamma(G)$. If the set of vertices traversed by the path $1 \xrightarrow{u} g$ is $1=k_1, k_2, \ldots, k_n=g$, then the embedded copy of $S\Gamma(u)$ is exactly the subgraph $\langle 1 \xrightarrow{u} g \rangle \cup \bigcup_{i=1}^n k_iS\Gamma$, where $\langle 1 \xrightarrow{u} g \rangle$ denotes the subgraph induced by $1 \xrightarrow{u} g$. In particular, since $S\Gamma(u)$ is an induced subgraph in $\Gamma(G)$ by \cite[Lemma 3.5]{Ste93}, it
 contains a $[uu^{-1}]_S \xrightarrow{w} [u]_S$ path if and only if each vertex of the $1 \xrightarrow{w} g$ path in $\Gamma(G)$ is contained in some $k_iS\Gamma$. 

For $h,k \in G$, observe that $h \in kS\Gamma$ if and only if $h \in kP$, which in turn is true if and only if $k \in hP^{-1}$. Denote the subgraph of $\Gamma(G)$ induced on $P^{-1}$ by $\ol {S\Gamma}$. Then $h \in kS\Gamma$ if and only if $k \in h\ol{S\Gamma}$. Thus we obtain the following lemma:

\begin{Lem}
	\label{lem:wlequ}
Let $u, w\in (X\cup X^{-1})^\ast$ label $1 \to g$ paths in $\Gamma(G)$. Then $[w]_S \geq [u]_S$ if and only if for each vertex $h$ traversed by $w$, there exists some vertex $k$ traversed by $u$ with $k \in h\ol{S\Gamma}$.
\end{Lem}

 We obtain the following statement as a consequence:

\begin{Thm}
	\label{thm:wmax}
Let $S$ be an $E$-unitary, special inverse monoid. A word $w\in (X\cup X^{-1})^\ast$ represents a maximum element in the $\sigma$-class $g$ if and only if for each vertex $h$ traversed by $1 \xrightarrow{w} g $, the graph	$h\ol{S\Gamma}$ either contains $1$ or $g$, or it is a cut set for $\Gamma(G)$ with $1$ and $g$ falling in different components. 

In particular, $S$ is $F$-inverse if and only if for every $g \in G$, there exists a $1 \xrightarrow{w} g$ path where $w$ satisfies the above property.
\end{Thm}

\begin{proof}
A word $w\in (X\cup X^{-1})^\ast$ represents a maximum element in the $\sigma$-class $g$ if and only if the condition of Lemma \ref{lem:wlequ} holds for all possible paths $1 \xrightarrow{u} g$. In particular for any vertex $h$ traversed by $w$, the graph $h\ol{S\Gamma}$ should intersect any $1 \to g$ path, which is by definition equivalent to the statement of the proposition.
\end{proof}

\section{Computable distortion}\label{sec:computability}

In this section we turn our attention to computational problems. In this context what matters is not so much whether the group distortion is bounded, as whether it is bounded by a recursive function. For simplicity we restrict attention to finitely generated inverse semigroups, with the word metric (and corresponding proper length function) induced by a fixed finite generating set. We expect that some of our results can be extended
to encompass countably generated inverse semigroups, but we refrain from attempting this here because computing with infinitely generated structures entails definitional subtleties which would significantly increase the length of the exposition and risk obscuring the key arguments.

\begin{Def}[recursively bounded group distortion]
	\label{def:recursive-group-dist}
Let $S$ be a finitely generated inverse semigroup with maximal group image $G$, and let $l_S$ and $l_G$ be the standard word metrics corresponding to some finite generating set of $S$ and the corresponding generating set of $G$. We say that a finitely generated inverse semigroup $S$ has \textit{recursively bounded group distortion} if there exists a recursive function
$\phi : \mathbb{N} \to \mathbb{N}$ such that for every $s \in S$ we have $l_S(s) \leq \phi(l_G( s\sigma))$ (or equivalently, $d_S(s,t) \leq \phi(d_G(s \sigma, t\sigma))$ for the respective distances for all $s,t \in S$, see Proposition \ref{prop:length-function-distortion}). 
\end{Def}

We emphasise that the above definition depends upon the use of a word metric, rather than an arbitrary standard metric as in the
previous section. We cannot usually expect algorithmic results to be independent of the choice of standard metric, since the coarse equivalence between different standard metrics may not have recursive bounding functions. We remark that, by essentially the same argument used to prove Proposition~\ref{prop:distortioncoarse}, recursively bounded group distortion is equivalent to the morphism $\sigma$ being a coarse embedding \textit{with recursive bounding functions}.

\begin{Ex}[group distortion can be bounded but not recursively bounded]
Let $H$ be a subgroup of a group $G$. Take $H'$ to be an isomorphic copy of $H$ disjoint from $G$, and $f \colon H' \to G$ an embedding with image $H$. Then we may consider an inverse monoid $S$ on the set $H' \sqcup G$ with multiplication defined by
$$gh = \begin{cases}
\textrm{the $H'$-product } gh &\textrm{ if } g, h \in H'; \\
\textrm{the $G$-product } gh &\textrm{ if } g, h \in G; \\
\textrm{the $G$-product } f(g)\ h &\textrm{ if } g \in H', h \in G; \textrm{ and} \\
\textrm{the $G$-product } g\ f(h) &\textrm{ if } g \in G, h\in H'.
\end{cases}$$
Then $S$ is an inverse monoid with identity element the identity of $H'$, the identity of $G$ being the only other
idempotent and the inverse of each element being its inverse in $H'$ or $G$ as appropriate. The maximal group image is clearly
$G$, with the natural morphism to it being defined by the map $f$ on $H'$ and the identity function on $G$. (We remark that $S$ is an example of
a \textit{Clifford semigroup}; see for example \cite[Section~4.2]{Howie} for further discussion of these.) 

Now suppose $G$ and $H$ are both finitely generated. Choose a finite generating set $X$ for $H$ and extend it to a finite
generating set $Y$ for $G$. Let $X'=\{x': x\in X\}$ denote a set in bijection with $X$, which we consider a generating set for $H'$ in the obvious way. Then $S$ is clearly generated by the set $X' \sqcup Y$. We let $d_S$ denote the corresponding word metric. Observe that since $[x]_G=[f(x')]_G$ for any $x \in X$, the word metric on $G$ is the same with respect to $Y$ as the one inherited from $S$, which we unambiguously denote by $d_G$. Consider furthermore the word metric $d_H$ on $H$ corresponding to the generating set $X$; identifying $H$ and $H'$, this coincides with the metric $d_S$ on $H'$ (but usually not with the metric $d_S$ on $H \subseteq G$).
We now have that $\sigma \colon S \to G$ restricts to an isometry on $G$, and
therefore
\begin{eqnarray*}
	\sup \{l_S(s): s\in S, l_G(s\sigma)\leq r\} &=& \sup \{l_S(h): s\in H', l_G(f(h))\leq r\}\\
	&=&\sup \{l_H(h): h\in H, l_G(h)\leq r\}
\end{eqnarray*}
which is finite for every $r$, since $l_G$ and $l_H$ both give rise to standard (i.e. proper and right invariant) metrics on the group $H$ which are then coarsely equivalent. 
Moreover, the group distortion of $S$ is bounded by exactly the same functions as the distortion of the embedding of $H$ into $G$.

If additionally $G$ and $H$ are both finitely presented, then $S$ is also finitely presented. Indeed, let 
$$G=\Gp \langle Y \mid w_1=1, \ldots ,w_n=1 \rangle,$$
$$H'=\Gp \langle X' \mid u_1=1, \ldots ,u_m=1 \rangle,$$
and let $e$ be symbol disjoint from $Y$ and $X'$. Consider the inverse monoid $M$ generated by $X' \sqcup Y \sqcup \{e\}$ subject to relations
\begin{itemize}
	\item[(R1)] $1=x'(x')^{-1}=(x')^{-1}x'$ for each $x \in X'$
	\item[(R2)] $e^2=e$
	\item[(R3)] $ex'=x'e=x$ for each $x  \in X$
	\item[(R4)] $e=yy^{-1}=y^{-1}y$ for each $y  \in Y$
	\item[(R5)] $u_1^2=u_1, \ldots, u_m^2=u_m$
	\item[(R6)] $w_1^2=w_1, \ldots ,w_n^2=w_n$.
\end{itemize}
It is straightforward to see using Stephen's algorithm that $S\Gamma(1)$ is isomorphic to the Cayley graph of $H'$ with respect to the generating set $X'$, and that $S\Gamma(e)$ is isomorphic to the Cayley graph of $G$ with respect to the generating set $Y \cup \{e\}$ where $[e]_G=1$. Furthermore it follows from (R1) that $[ww^{-1}]_M=1$ for every word $w$ in the alphabet $X' \cup (X')^{-1}$, and similarly it follows from (R3) and (R4) that $[ww^{-1}]_M=[e]_M$ for every word $w$ in $X' \cup (X')^{-1} \cup Y \cup Y^{-1}$ which involves at least one letter from $Y \cup Y^{-1}$. This means that $1$ and $[e]_M$ are the only idempotents of $M$, which is therefore a union of its two maximal subgroups $V(S\Gamma(1)) \cong H'$ and $V(S\Gamma(e)) \cong G$. Observe furthermore that identifying $M$ with $H' \sqcup G$ and defining $f(h):=eh$ for $h \in H'$ results by (R3) in the exact same multiplication rules for $M$ as the ones defining $S$, yielding $S \cong M$.

It is well known \cite[Proposition~2.1]{Farb} that if $G$ is finitely presented with decidable word problem, then the distortion of the embedding of $H$ in $G$ is bounded by a recursive function if and only if $H$ has decidable membership in $G$. Thus, by taking $G$ to be a finitely presented group having a
finitely presented subgroup $H$ with undecidable membership problem (an example of which can be found in \cite[Theorem 1.2]{Bridson}), the above construction yields an inverse monoid with group distortion
which is bounded but not recursively bounded.
\end{Ex}

Recall that an inverse semigroup is $F$-inverse if in every $\sigma$-class there exists a maximal element; in a computational setting, the natural question is not so much whether a maximal element exists, as whether it can be effectively computed. This observation
leads to the following definition:
\begin{Def}[effectively $F$-inverse]
A finitely generated inverse semigroup $S = \langle X \rangle$ is called \textit{effectively $F$-inverse} if it is $F$-inverse, and there is an algorithm which on input a word $w \in (X \cup X^{-1})^\ast$, outputs a representative of the maximal element $[w]_S^\m$.
\end{Def}

\begin{Prop}\label{prop:effectivefinverse}
Suppose $S$ is a finitely generated inverse semigroup which is effectively $F$-inverse, and whose maximal group
image has decidable word problem. Then $S$ has recursively bounded group distortion.
\end{Prop}
\begin{proof}
Since $S$ is effectively $F$-inverse, there is a computable function, $f$ say, which maps each word $w$ to an upper bound on $l_S([w]^\m)$; indeed one may compute a word representing $[w]^\m$ and then set $f(w)$ to be the number of letters in this word, which by sub-additivity of the length function exceeds the length of the element represented.

Define $\phi : \mathbb{N} \to \mathbb{N}$ by
$$\phi(n) = \max \lbrace f(w) \mid |w| \leq n \rbrace.$$
Clearly the function $\phi$ is computable, since $f$ is computable and we can enumerate the finitely many words $w$ of 
length less than $n$.\footnote{We remark that this step may not work for length functions other than those arising from word metrics.
} And for every $s \in S$ there exists a word $w$
representing $s\sigma \in G$ of length $l_G(s\sigma)$, whereupon 
$$l_S(s) \ \leq \ l_S(s^\m) \ = \ l_S([s\sigma]^\m) \ \leq \ f(w) \ \leq \ \phi(l_G(s\sigma))$$
where the first inequality holds because $s \leq s^m$, the equality holds because $s^\m = [s\sigma]^\m$, the second inequality holds by the definition of $f(w)$ and the final inequality holds by the definition of $\phi$.
\end{proof}

The following is an effective version of Lemma~\ref{lem:sprawling}:
\begin{Lem}
	\label{lem:sprawlingeffective}
Let $S = \langle X \rangle$ be a finitely generated $E$-unitary inverse monoid equipped with the word metric and having  sprawling Schützenberger graph of $1$, and whose maximal group image has decidable word problem. Then $S$ is effectively $F$-inverse, and hence has recursively bounded group distortion.
\end{Lem}
\begin{proof}
As in the proof of Lemma~\ref{lem:sprawling}, for each $g \in G$ if we choose $w_g$ to label a path from $1$ to $g$ in the (necessarily connected) graph  $\Sg \cup g\Sg$, then $[w_g]_S$ is the greatest element of $\sigma^{-1}(g)$. So
it suffices to show that we can find such a path. If $G$ has a decidable word problem, then in particular the ball of each radius $r$ in the
Cayley graph of $G$ is effectively constructible, and for any $g$, approximating $\Sg \cup g\Sg$ by Stephen's procedure in larger and larger portions of the group Cayley graph will eventually find some choice of $w_g$. The fact that $S$ has
recursively bounded group distortion then follows from Proposition~\ref{prop:effectivefinverse}.
\end{proof}

This leads immediately to an effective version of Proposition~\ref{prop:Finverse-suff2}:
\begin{Prop}
	\label{prop:Finverse-suffeffective}
Let $P$ be the prefix monoid associated to the $E$-unitary special inverse monoid $S=\Inv\langle X \mid w_i=1\ (i \in I) \rangle$. Suppose the maximal group image $G$ of $S$ has decidable word problem. If $P$ is right reversible, or
if $G$ is the group of right quotients of $P$, then $S$ is effectively $F$-inverse and has recursively bounded group distortion.
\end{Prop}
\begin{proof}
If $G$ is the group of right quotients of $P$ then by Proposition~\ref{prop:Finverse-suff} $P$ is right reversible,
so it suffices to consider the latter case.

As in the proof of Proposition~\ref{prop:Finverse-suff2}, let $Z$ be the subset of generators occurring in the relations and
put $Y=X \setminus Z$ so that $S$ splits as a free product the $F$-inverse monoids $\FIM(Y)$ and $T=\Inv\langle Z \mid w_i=1\ (i \in I)\rangle$. The monoid $T$ satisfies the conditions of Proposition~\ref{prop:Finverse-suff}, and therefore is
$F$-inverse with sprawling Sch\"utzenberger graph of $1$. The maximal group image $H$ of $T$ is a subgroup of $G$, and therefore also has decidable word problem. Hence by Lemma~\ref{lem:sprawlingeffective}, $T$ is effectively $F$-inverse. 

The maximal group image $G$ of $S$ is the free product of $\FG(Y) $ and $H=\Gp\langle Z \mid w_i=1\ (i \in I) \rangle$. By Lemma \ref{lem:free-product} and the observation thereafter, given $g \in G$ written as an alternating product $g_1 \cdots g_n$ of non-identity elements of $\FG(Y)$ and $H$, the greatest element in the $\sigma$-class $g$ of $S$ is $g_1^\m \cdots g_n^\m$, where $g_i^\m$ is the greatest element in the $\sigma$-class $g_i$ of $\FIM(Y)$ or $T$. In particular, if $u_i$ is a representative of $g_i^\m$, then $u_1 \ldots u_n$ is a representative of the greatest element in $g$. 

Let $w \in (X \cup X^{-1})^\ast$, and write $w$ as $w_1 \ldots w_n$ where each factor $w_i$ belongs either to $(Y \cup Y^{-1})^+$ or $(Z \cup Z^{-1})^+$, and represents a non-identity element in $\FG(Y)$, respectively $H$. If $w_i \in (Y \cup Y^{-1})^+$, then a representative $u_i$ of $[w_i]_{\FIM(Y)}^\m$ is the freely reduced form of $w_i$, which is effectively computable. If $w_i \in (Z \cup Z^{-1})^+$, then a representative $u_i$ of $[w_i]_{T}^\m$ is computable as $T$ is effectively $F$-inverse. This yields a computable representative $u=u_1 \ldots u_n$ of $[w]^\m_S$, and so $S$ is effectively $F$-inverse. By Proposition~\ref{prop:effectivefinverse}, $S$ also has recursively bounded group distortion.
\end{proof}

We do not know if the hypotheses of Lemma~\ref{lem:sprawlingeffective} or Proposition~\ref{prop:Finverse-suffeffective} yield a decidable word problem in $S$ -- we pose this as Question \ref{Q:solvable-if-P-right-rev} below.

For the following result we are grateful to John Meakin who (in a private communication) established a special case of it and in doing so provided the idea of the proof for the general case we state here.
\begin{Thm}
Let $S$ be a finitely generated $E$-unitary special inverse monoid with recursively bounded group distortion and such that the maximal group image has decidable membership problem. Then the following four decision problems are all Turing equivalent:
\begin{enumerate}
\item \label{item:recgroupdist-solvablewp} the word problem for the monoid $S$;
\item \label{item:recgroupdist-decidableif1} the problem of deciding whether a given word represents $1$ in the monoid $S$;
\item \label{item:recgroupdist-decidableifrightunit} the problem of deciding whether a given word represents an element of the $\mathcal{R}$-class of $1$ in the monoid $S$; and
\item \label{item:recgroupdist-prefixmembership} the membership problem for the prefix monoid of $S$ in the maximal group image.
\end{enumerate}
\end{Thm}
\begin{proof}
Clearly given an oracle to decide the word problem in $S$ one can check if a given word represents $1$ in $S$, so (\ref{item:recgroupdist-decidableif1}) is trivially reducible to (\ref{item:recgroupdist-solvablewp}). For any element $s \in S$ we have $s \rrel 1$ if and only if $ss^{-1} = 1$; so if given an oracle for (\ref{item:recgroupdist-decidableif1}) we can check if a given word $w$ represents an element of the $\mathcal{R}$-class of $1$ by checking if the word $w w^{-1}$ represents $1$. The fact that (\ref{item:recgroupdist-solvablewp}) reduces to (\ref{item:recgroupdist-prefixmembership}) is established in the proof\footnote{The statement
of \cite[Theorem 3.2]{IMM} discusses only the decidable case, but the proof exhibits a Turing reduction.} of \cite[Theorem 3.2]{IMM}. What remains is to show that (\ref{item:recgroupdist-prefixmembership}) reduces to (\ref{item:recgroupdist-decidableifrightunit}).

Suppose, then, that we are given an oracle to decide membership of the $\mathcal{R}$-class of $1$, and that we are given a word $w$. Let $G$ be the maximal group image of $S$ and $g$ be the element of $G$ represented by $w$. Since the prefix monoid in $G$ is the image under the natural morphism of the $\mathcal{R}$-class of $1$ in $S$, the element $g$ lies in the prefix monoid if and only if it is the image under the natural morphism of some (necessarily unique, since $S$ is $E$-unitary) element of the $\mathcal{R}$-class of $1$. Since the distortion of the morphism is computable, the length of any such element is bounded above by a computable function of the length of $g$, and hence by a computable function of the length of
$w$. It clearly suffices to compute such a bound, enumerate all words up to that length, and for each word check if
the word represents $g$ in $G$ (which we can do since $G$ has decidable word problem) and if it represents an
element of the $\mathcal{R}$-class of $1$ in $S$ (which we can do using the oracle for (\ref{item:recgroupdist-decidableifrightunit})).
\end{proof}

One might hope that in the realm of $E$-unitary finitely presented special inverse monoids, recursively bounded group distortion together with a solvable word problem in the maximal group image yields a solvable word problem. This is not true: in fact we shall show that counterexamples include the one-relator inverse monoid with undecidable word problem given in \cite[Theorem 3.9]{Bob}, and the class of $E$-unitary inverse semigroups with hyperbolic Sch\"utzenberger graphs and undecidable word problem constructed in \cite[Section 5]{GSSz}.

All of these constructions have the following form: let $X=\{x_1, \ldots, x_n\}$, and let $r_1, \ldots, r_m,$ $s_1, \ldots, s_k$ be words in $(X \cup X^{-1})^*$. Let $t$ be a disjoint letter from $X$, and define $e$ as the word
$$e=\prod_{i=1}^n x_i x_i^{-1} \prod_{i=1}^m  (ts_it^{-1})(ts_it^{-1})^{-1} \prod_{i=1}^n x_i^{-1} x_i,$$
and put
$$S=\Inv \langle X, t \mid er_1=1, r_2=1, \ldots, r_m=1 \rangle.$$

It was shown in \cite[Theorem 3.8]{Bob} that $S$ is $E$-unitary with maximal group image $G \ast \FG(t)$ where $G=\Gp\langle X \mid r_1=1, \ldots, r_m=1 \rangle$, and that if $S$ has solvable word problem then the membership problem of the submonoid generated by $s_1, \ldots, s_k$ in $G$ is decidable. This is used in \cite{Bob} and \cite{GSSz} to construct $E$-unitary special inverse monoids with unsolvable word problem, from groups with undecidable submonoid membership problem.

Let $\Gamma$ be the Cayley graph of $G \ast \FG(t)$ with respect to the generating set $X \cup \{t\}$, and consider the Sch\"utzenberger graphs of $S$ with respect to the same generating set. Then the following holds.

\begin{Prop}
	\label{prop:sigmaiso}
The inverse monoid $S$ as defined above is $F$-inverse, and has the property that every Sch\"utzenberger graph isometrically embeds into $\Gamma$ via $\sigma$ (so in particular, has recursively bounded group distortion).
\end{Prop}

\begin{proof}
Take any Sch\"utzenberger graph of $S$, and denote its embedded copy in $\Gamma$ by $\Delta$.	
	
First, observe that since $\Delta$ is $P$-complete, one can read the word $e$, and thus the letters $x_i$ and $x_i^{-1}$ ($i=1, \ldots,n$) from any vertex in $\Delta$. In particular, if $w \in (X \cup X^{-1})^\ast$, one can read $w$ in $\Delta$ from any vertex. 

Let $u, v \in V(\Delta)$, and suppose that $w \in (X \cup X^{-1} \cup \{t\} \cup \{t^{-1}\})^*$ labels a $u \to v$ geodesic path in $\Gamma$. We will show that this path is also contained in $\Delta$.
Write $$w=t^{k_1}w_1t^{k_2}\ldots t^{k_l}w_l t^{k_{l+1}},$$ 
where $w_i \in (X \cup X^{-1})^+$, $k_i \in \mathbb Z$ and $k_1, \ldots, k_l \neq 0$. Observe that since $w$ is assumed to label a geodesic path, $[w_i]_G \neq 1$ for any $i$.

Since $\Delta$ is connected, there exists a $u \to v$ path in $\Delta$, with label  $w' \in (X \cup X^{-1} \cup \{t\} \cup \{t^{-1}\})^*$. Then $[w']_{G \ast \FG(t)}=[w]_{G \ast \FG(t)}$. We can without loss of generality assume that the $w'$ labels a simple path, i.e. $w'$ contains no subword with value $1$ in $G \ast \FG(t)$. Put
$$w'=t^{k'_1}w'_1t^{k'_2}\ldots t^{k'_l}w'_j t^{k'_{j+1}}$$
where $w'_i \in (X \cup X^{-1})^+$, $k'_i \in \mathbb Z$ and $k'_1, \ldots, k'_l \neq 0$. Then since $w$ and $w'$ have the same value in $G \ast \FG(t)$, and $[w'_i]_G \neq 1$ for any $i$, we must have $l=j$, $k_i=k'_i$ and $[w_i]_G=[w'_i]_G$ for all $i$. 
Let $\alpha_0, \ldots, \alpha_l$ and $\beta_{1}, \ldots, \beta_{l+1}$ be the vertices in $\Gamma$ traversed by the path $u \xrightarrow{w} v$ at the following points:
$$u=\alpha_0 \xrightarrow{t^{k_1}} \beta_1 \xrightarrow{w_1} \alpha_1 \xrightarrow{t^{k_2}} \beta_2 \xrightarrow{w_2} \cdots \xrightarrow{t^{k_l}} \beta_{l} \xrightarrow{w_l} \alpha_l  \xrightarrow{t^{k_{l+1}}} \beta_{l+1}=v.$$
It follows that the path $u \xrightarrow{w'} v$ also traverses the same vertices, factorizing as
$$u=\alpha_0 \xrightarrow{t^{k_1}} \beta_1 \xrightarrow{w'_1} \alpha_1 \xrightarrow{t^{k_2}} \beta_2 \xrightarrow{w'_2} \cdots \xrightarrow{t^{k_l}} \beta_{l} \xrightarrow{w'_l} \alpha_l  \xrightarrow{t^{k_{l+1}}} \beta_{l+1}=v.$$
Since $w'$ lies in $\Delta$, in particular we have $\beta_i \in V(\Delta)$ for all $i$, and thus the paths $\beta_i \xrightarrow{w_i} \alpha_{i}$ must all lie in $\Delta$, which shows that $w$ is contained in $\Delta$ indeed.

This proves that $\sigma$ is an isometric embedding with respect to the word metrics in $S$ and $G$, so $\phi$, as defined in Definition \ref{def:recursive-group-dist} can be chosen to be the identity map, which is computable. Since $S$ is $E$-unitary and special, by Theorem \ref{thm:uniformcoarse}, it also must be $F$-inverse.
\end{proof}

We have the following theorem as a corollary of the above and \cite[Theorem 3.9]{Bob}.

\begin{Thm}
	\label{BobsFinv}
There exists a one-relator special inverse monoid which is $F$-inverse, has undecidable word problem, has maximal group image with decidable word problem, and is such that $\sigma$ isometrically embeds each $\rrel$-class into the maximal group image.
\end{Thm}

We remark that these examples never have right reversible prefix monoids. Indeed, it is not difficult to check that if $g \in G$ is such that $g^{-1}$ is not contained in the subsemigroup generated by $s_1, \ldots, s_k$, and $h \in G\setminus\{1\}$, then the principal right ideals $tgP$ and $htgP$ do not intersect. 

We conclude with some questions which remain open.

\begin{Ques}
	\label{Q:solvable-if-P-right-rev}
If $S$ is an $E$-unitary finitely presented special inverse monoid such that the prefix monoid $P$ of $S$ is right reversible and the word problem is solvable in $S / \sigma$, is the word problem solvable in $S$?
\end{Ques}

By Proposition \ref{prop:Finverse-suff}, if $P$ is right reversible, then $S/\sigma$ splits as a free product of a free group and the group of right quotients of $P$. By \cite{IMM}, $S$ has a solvable word problem if $P$ has a decidable membership problem in $S/\sigma$, which holds if and only if it has decidable membership problem in $PP^{-1}$. Thus a positive answer to Question \ref{Q:solvable-if-P-right-rev} would be implied by a positive answer to the following:

\begin{Ques}
	\label{Q:P-dec-in-PPinv}
	If $G$ is the group of right quotients of a right reversible cancellative monoid $P$ and $G$ has solvable word problem, is the membership problem of $P$ in $G$ decidable?
\end{Ques}

If an $E$-unitary, special inverse monoid is $F$-inverse due to satisfying the conditions of Proposition \ref{prop:Finverse-suff} or \ref{prop:Finverse-suff2} and its maximal group image has a solvable word problem, then it is also effectively $F$-inverse by Proposition~\ref{prop:Finverse-suffeffective}. We ask whether this is true for all (special) $F$-inverse monoids.

\begin{Ques}
	\label{Q:F-invmaxdec}
If $S=\Inv\langle X \mid u_1=v_1, \ldots, u_n=v_n \rangle$ is $F$-inverse, and the word problem in $S/\sigma$ is decidable, is there an algorithm that on input $w \in (X \cup X^{-1})^\ast$ outputs a representative for $[w]_S^\m$? What about if $S$ is special?
\end{Ques}

The examples in Section \ref{sec:computability} do satisfy this property: it follows from the proof of Proposition \ref{prop:sigmaiso} that for any $g \in G \ast \FG(t)$, the greatest element of $S$ in the $\sigma$-class $g$ can be represented by any element of the form $w=t^{k_1}w_1t^{k_2}\ldots t^{k_l}w_l t^{k_{l+1}}$, 
where $[w]_{G \ast \FG(t)}=g$, 
$w_i \in (X \cup X^{-1})^+$, $k_i \in \mathbb Z$, $k_1, \ldots, k_l \neq 0$, and $[w_i]_G \neq 1$ for any $i$, since such a word is always readable from $1$ in a Sch\"utzenberger graph containing $1$ and $g$. If the word problem in $G$ is solvable, than any word in $(X \cup X^{-1} \cup \{t\} \cup \{t^{-1}\})^*$ representing $g$ can be effectively reduced to one of the form of $w$.

Observe that in the special case, a positive answer to Question \ref{Q:F-invmaxdec} would be implied if the conditions in Theorem \ref{thm:wmax} were algorithmically verifiable, that is, if there was an algorithm that on input $g,h \in G$, would eventually verify if no $1 \to g$ path existed in $\Gamma(G)$ avoiding $h\ol\Sg$ (but need not terminate on other inputs).

\end{document}